\documentclass[twocolumn]{autart}

\usepackage{cite, comment}
\usepackage{mathrsfs}
\usepackage{graphicx}
\usepackage{graphics}
\usepackage{amssymb}
\usepackage{mathtools}
\usepackage{color}
\usepackage{xspace}
\usepackage{algpseudocode}
\usepackage{bbm}
\usepackage{algorithmicx}
\usepackage{url}
\usepackage{algorithm}
\usepackage{soul}
\usepackage{rotating}
\usepackage{chngcntr}
\usepackage{apptools}
\usepackage{enumitem}
\usepackage{listings}

\newcommand{\oprocendsymbol}{\hbox{$\bullet$}}
\newcommand{\oprocend}{\relax\ifmmode\else\unskip\hfill\fi\oprocendsymbol}

\newcommand{\longthmtitle}[1]{\mbox{}\textit{{(#1):}}}

\newcommand{\real}{{\mathbb{R}}}
\newcommand{\reals}{{\mathbb{R}}}

\newcommand{\eps}{\epsilon}

\newcommand{\blue}[1]{{\color{black}#1}}

\newcommand{\vect}[1]{\boldsymbol{\mathbf{#1}}}
\newcommand{\vectsf}[1]{\boldsymbol{\mathbf{\mathsf{#1}}}}

 \newcommand{\boxend}{\hfill \ensuremath{\Box}}

\newtheorem{assump}{Assumption}[section]

\newtheorem{thm}{Theorem}[section]
\newenvironment{proof}[1][Proof]{\noindent\textbf{#1.} }{\hfill $\square$}

\newtheorem{rem}{Remark}[section]

\newtheorem{lem}{Lemma}[section]



\allowdisplaybreaks
\begin{document}

\begin{frontmatter}
   \runtitle{}
  
   \title{Time-Varying Convex Optimization with O(n) Computational Complexity}
  \thanks[footnoteinfo]{A preliminary  version of this paper was presented in~\cite{MR-SSK:24}. 
This work was supported by NSF CAREER award ECCS 1653838. Corresponding author: M. Rostami. }
      \author[Paestum]{Mohammadreza Rostami}\ead{mrostam2@uci.edu}\quad
  \author[Paestum]{Solmaz S. Kia}\ead{solmaz@uci.edu}

  \address[Paestum]{Department of Mechanical and Aerospace
    Engineering, University of California, Irvine}

  \begin{keyword}
 time-varying optimization, convex optimization,  information streaming
  \end{keyword}

  \begin{abstract}
   In this article, we consider the problem of unconstrained time-varying convex optimization, where the cost function changes with time. We provide an in-depth technical analysis of the problem and argue why freezing the cost at each time step and taking finite steps toward the minimizer is not the best tracking solution for this problem. We propose a set of algorithms that by taking into account the temporal variation of the cost aim to reduce the tracking error of the time-varying minimizer of the problem. The main contribution of our work is that our proposed algorithms only require the first-order derivatives of the cost function with respect to the decision variable. This approach significantly reduces computational cost compared to the existing algorithms, which use the inverse of the Hessian of the cost. Specifically, the proposed algorithms reduce the computational cost from $O(n^3)$ to $O(n)$ per timestep, where $n$ is the size of the decision variable. Avoiding the inverse of the Hessian also makes our algorithms applicable to non-convex optimization problems. We refer to these algorithms as $O(n)$-algorithms. These $O(n)$-algorithms are designed to solve the problem for different scenarios based on the available temporal information about the cost. We illustrate our results through various examples, including the solution of a model predictive control problem framed as a convex optimization problem with a streaming time-varying cost function.
  \end{abstract}
\end{frontmatter}

\section{Introduction}
There is growing interest in optimization and learning problems where traditional algorithms cannot match the pace of streaming data due to limitations in computational resources~\cite{ MF-SP-AR:17, AS-AM-AK-GL-AR:16, ED-AS-SB-LM:19, AS-ED-SP-GL-GBG:20, esteki2022distributed, kim2024mathcal}. 
Applications like power grids, networked autonomous systems, real-time data processing, learning methods, and data-driven control systems benefit from adaptive optimization algorithms. This paper aims to advance knowledge on such algorithms for problems with time-varying cost functions. We focus on unconstrained convex optimization where the objective function is time-varying. Specifically, for a decision variable $\vect{x}\in\real^n$ and time $t\in\real_{\geq0}$, the objective function is $f(\vect{x},t):\real^n\times\real_{\geq 0}\rightarrow\real$ and the problem is \begin{align}\label{eq::opt} \vectsf{x}^{\star}(t) &= \arg\underset{{\vect{x}\in \reals^n}}{\min} \,\,f(\vect{x},t),\quad t\in\real_{\geq0}, \end{align} where $\vectsf{x}^\star(t)$ is the global minimizer at time $t$. In what follows, we denote $\vectsf{x}^\star(t)$ by $\vectsf{x}^\star_t$. Figure~\ref{fig:illustration_page1} illustrates a time-varying cost and its  $t\mapsto\vectsf{x}^\star_t$.  
For $t\mapsto\vectsf{x}^\star_t$ to be  a trajectory of a solution of problem~\eqref{eq::opt} is~\cite{SR-WR:17} \begin{align}\label{eq::KKT_timeVar} \nabla_{\vect{x}}f(\vectsf{x}_t^\star,t)=\vect{0},~~~\quad t\in\real_{\geq0}. \end{align}

 \begin{figure}[t]
  \centering
    \includegraphics[scale=0.09]{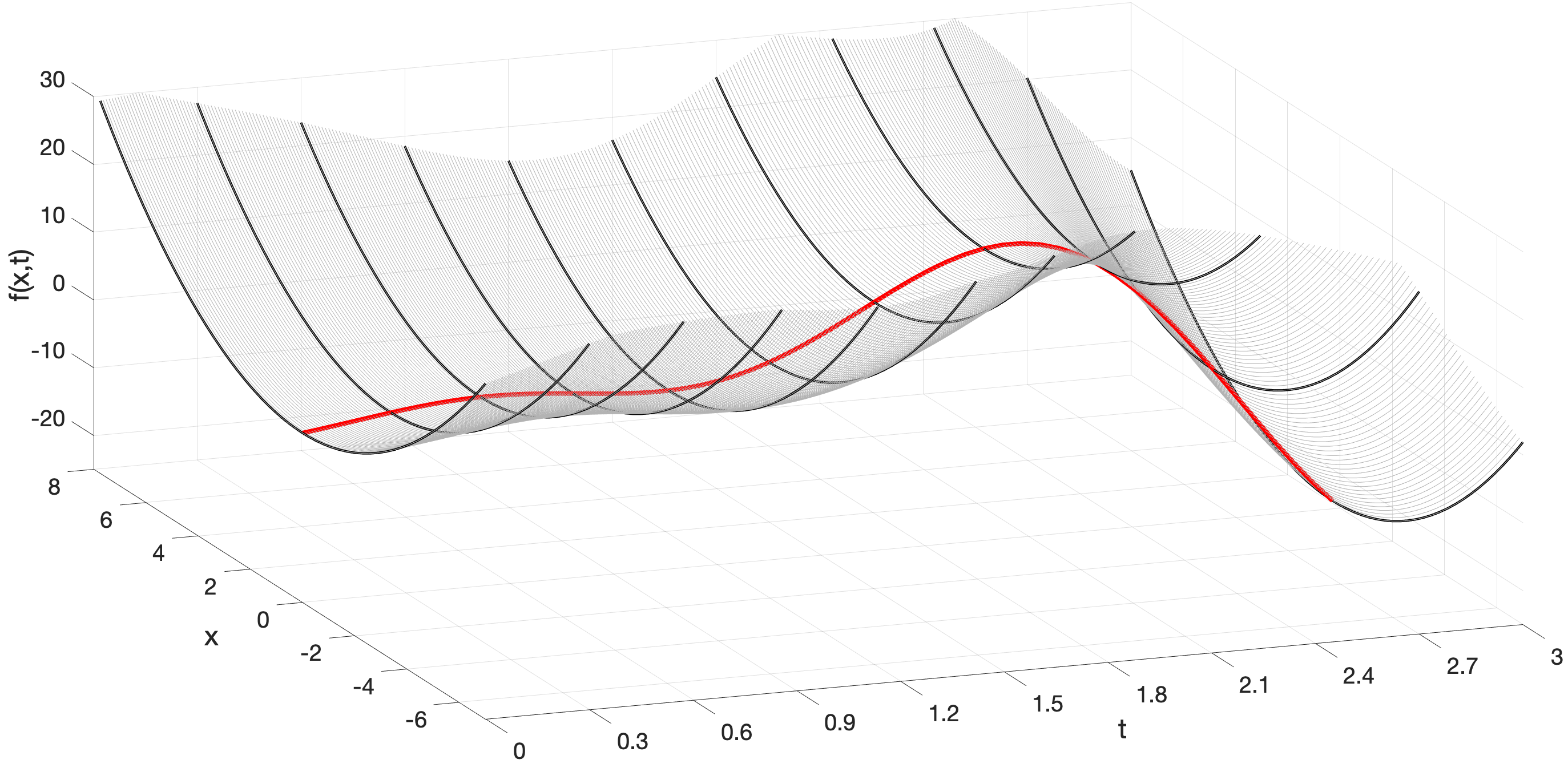}
      \caption{\small{A time-varying $f(\vect{x},t)$ vs. $\vect{x}$ and $t$ (gray plot) and the trajectory of $f(\vectsf{x}^\star_t,t)$ vs. $\vectsf{x}^\star(t)$ and $t$ (red curve). The darker contours are added to improve visualization of the function and does not signify any particular property of the function.}}\label{fig:illustration_page1}
\end{figure}

A simple way to solve~\eqref{eq::opt} is by sampling at specific times $t_k$ and solving a sequence of optimization problems assuming \blue{$f(\vect{x},t)=f(\vect{x},t_k)$} for each interval $t\in[t_k,t_{k+1})$. However, this leads to steady-state tracking errors, significant if the cost function varies quickly over time. To address the limitations of this approach, optimization algorithms specifically designed for time-varying optimization problems have been emerging recently.

\emph{Literature review}\\~\\
Algorithms for time-varying optimization take into account the temporal variation of the cost and aim to track the optimal trajectory $t\mapsto\vectsf{x}^\star_t$ asymptotically over time. 
 Some of these algorithms draw inspiration in their design and analysis from trajectory tracking control theory approaches. 
For example,~\cite{PCV-SJ-FB:22} and \cite{DHN-LVT-TK-SJJ:18} have  studied solving continuous-time time-varying convex optimization problems from the
\emph{contraction theory} perspective and has provided a tracking error bounds between any solution trajectory and the equilibrium trajectory. Recent work~\cite{AD-VC-AG-GR-FB:23} shows that for any contracting dynamics dependent on parameters, the tracking error is uniformly upper-bounded in terms of the contraction rate, the Lipschitz constant, and the rate of change of the parameter. The proposed algorithms in \cite{DHN-LVT-TK-SJJ:18,PCV-SJ-FB:22,AD-VC-AG-GR-FB:23} are continuous flows and their convergence analysis has been carried out in continuous~time. To guarantee contraction in the discrete-time domain, the discretization step size must be designed carefully based on the system's dynamics, which is not a trivial task.

For twice differentiable costs, derivative of~\eqref{eq::KKT_timeVar} leads to  
 \begin{align}\label{eq::x_star_trajectory}
 \dot{\vectsf{x}}_t^\star = -\nabla_{\vect{xx}} f(\vectsf{x}_t^\star,t) ^{-1}\nabla_{\vect{x}t} f(\vectsf{x}_t^\star,t),\quad t\in\real_{\geq 0},
 \end{align}
Thus, recent work such as~\cite{MF-SP-AR:17,SR-WR:17,ECH-RMW:15,YD-JL-MA:21,OR-MB:20,kim2024mathcal,marchi2024framework} propose continuous-time algorithms that converge asymptotically to trajectories satisfying~\eqref{eq::x_star_trajectory} from any initial conditions. Prediction-update approaches are proposed in~\cite{AS-AM-AK-GL-AR:16,ED-AS-SB-LM:19,NB-AS-RC:20,NB:21} for discrete-time implementations of these continuous-time algorithms, but these algorithms converge only to a neighborhood of the optimal trajectory. Despite promising performance, these existing algorithms use second-order derivatives of the cost (an $O(n^2)$ computation) and similar to~\eqref{eq::x_star_trajectory}, they also require the inverse of the Hessian (an $O(n^3)$ computation), adding to the computational cost. Therefore, although elegant, all these approaches suffer from high computational complexity. Computing the inverse of the Hessian also limits the use of these algorithms for non-convex optimization problems where the inverse may not exist. The use of the Hessian has also proven to be restrictive in designing distributed optimization algorithms inspired by these second-order methods. For example, the algorithms in~\cite{SR-WR:17} and~\cite{BH-YZ-ZM-WR:20} require that Hessians of all local objective functions be~identical.

Another body of literature relevant to time-varying optimization is \emph{online optimization}, as introduced by Zinkevich~\cite{MZ:03} and extensively studied~\cite{EH:16,KL-HX:23,PZ-YZ-LZ-ZZ:24, mokhtari2016online}. Although there may be some conceptual similarities between these two methods, key distinctions exist. In online optimization, a player iteratively makes decisions without knowing the outcomes in advance. Each decision results in a cost, which can vary and may depend on the actions taken by the player~\cite{EH:16}. In particular, online optimization involves an infinite sequence of convex functions $\{f_1,f_2,\cdots\}$, where at each time step $k\in\{1,\cdots,T\}$, an algorithm/player selects a vector $\vect{x}_k$ and then receives the corresponding cost function $f_k$. The cost sequence is not known in advance, and the cost functions are independent and uncorrelated. Rather than minimizing the current cost $f_k$, the focus is on minimizing regret, which is defined as the difference between the cumulative cost incurred by the chosen $\vect{x}_k$ over a time horizon $T$ and the cost that would have been incurred by the optimal sequence of decisions in hindsight. In contrast, time-varying optimization considers a cost function that evolves over time and assumes that this evolution is known or can be estimated. The goal is to track the time-varying $t\mapsto\vectsf{x}^\star_t$ by leveraging temporal information such as time derivatives of the cost function, including $\nabla_t f$, $\nabla_{tt} f$, and $\nabla_{xt} f$. This allows for efficient tracking and adaptation to changes in the cost~landscape.

\emph{Statement of contribution}\\~\\
We introduce novel discrete-time algorithms for solving optimization problem~\eqref{eq::opt}. Our contribution lies in using only first-order derivatives of the cost function with respect to $\vect{x}$, reducing computational cost from $O(n^3)$ to $O(n)$ per timestep compared to Hessian-based methods. These $O(n)$-algorithms are applicable to nonconvex problems and are designed for various scenarios based on available temporal cost information. Our algorithms have a prediction-update structure, where the goal is for the trajectory $t_k\mapsto\vect{x}_k$ to follow $t_k\mapsto\vectsf{x}^\star_k$ asymptotically over time. The prediction step uses the temporal information about the cost to transition $\vect{x}_k$ to a predicted $\vect{x}_{k+1}^-$ at time $t_{k+1}$ such that $f(\vect{x}_{k+1}^-,t_{k+1})\leq f(\vect{x}_{k},t_{k+1})$. Then, in the update step, a gradient descent step is taken in the time step $t_{k+1}$ to generate $\vect{x}_{k+1}$. The rigorous analysis shows that, similar to the second-order discrete-time algorithm, our proposed algorithms achieve tracking with a nonzero steady-state error. Our results also include a hybrid algorithm, which employs a second-order prediction step whenever a first-order prediction is skipped due to the internal conditions of the algorithm to avoid numerical issues. We demonstrate our results through two examples, including the solution of a Model Predictive Control (MPC) problem framed as a convex optimization problem with a streaming  time-varying cost function. To maintain the focus on the com-
position of the algorithms, we have moved the proofs of
our formal convergence analysis to the appendix.

\section{Preliminaries}\label{sec::prelim}
Our notation is defined as follows. The set of real numbers, nonnegative real and positive real numbers are, respectively, $\real$, $\real_{\geq0}$ and $\real_{>0}$. For a vector $\vect{x}\in\reals^n$, the Euclidean  norm is  $\|\vect{x}\|\!=\!\sqrt{\vect{x}^\top\vect{x}}$.  
The partial derivatives of a function $f(\vect{x},t):\real^n\times\real_{\geq 0}\rightarrow\real^n$ with respect to $\vect{x}\in\real^n$ and $t\in\real_{\geq 0}$ are given by $\nabla_{\vect{x}}f(\vect{x},t)$ (referred to hereafter as `gradient') and $\nabla_{t}f(\vect{x},t)$; the second order derivatives are denoted by $\nabla_{\vect{x}t}f(\vect{x},t)$,  $\nabla_{\vect{x}\vect{x}}f(\vect{x},t)$ and $\nabla_{tt}f(\vect{x},t)$. 

For a  twice differentiable function $f(\vect{\chi},\tau):\real^n\times\real_{\geq 0}\rightarrow\real^n$, Taylor series expansion~\cite{DPB:99} reads as
\begin{align}\label{eq::exact_taylor}
 &f(\vect{\chi}_{k+1},\tau_{k+1})= f(\vect{\chi}_{k},\tau_{k})\nabla_{\vect{\chi}} f(\vect{\chi}_{k},\tau_k)^\top(\vect{\chi}^-_{k+1}-\vect{\chi}_{k})\,+\nonumber\\&~~~~~+\nabla_{\tau} f(\vect{\chi}_{k},\tau_k)(\tau_{k+1}-\tau_k)+\frac{(\tau_{k+1}-\tau_k)^2}{2}\nabla_{\tau\tau} f(\vect{\zeta},\theta)\nonumber\\&~~~~~+\frac{1}{2}(\vect{\chi}^-_{k+1}-\vect{\chi}_{k})^\top \nabla_{\vect{\chi\chi}}f(\vect{\zeta},\theta)(\vect{\chi}^-_{k+1}-\vect{\chi}_{k})\nonumber\\&~~~~\quad+\nabla_{\vect{\chi}\tau} f(\vect{\zeta},\theta)^\top(\vect{\chi}^-_{k+1}-\vect{\chi}_{k})(\tau_{k+1}-\tau_k),
\end{align}  
where $\vect{\vect{\zeta}}\!\in\![\vect{\chi}_k,\vect{\chi}_{k+1})\subset\real^n$  and $\theta\!\in\![\tau_k,\tau_{k+1})\subset\real_{\geq0}$. 

Next, we present the assumptions that we will invoke in analysis of our proposed algorithms to solve optimization problem~\eqref{eq::opt} and discuss the corresponding properties that arise from these assumptions. We begin by specifying the conditions for the steaming cost to be well-posed.
\begin{assump}\longthmtitle{Well-posedness conditions}\label{asm:well_pose}{\rm
    At any time $t\in\real_{\geq 0}$ and any finite $\vect{x}\in\real^n$, the cost function of~\eqref{eq::opt} satisfies  $|f(\vect{x},t)|<\infty$. Moreover, at any $t\in\real_{\geq 0}$, optimization problem~\eqref{eq::opt} is solvable and its minimum value is finite, i.e., $|f(\vectsf{x}^\star_t,t)|=\mathsf{f}_t^\star<\infty$.}\boxend
\end{assump}
Next, we present a set of assumptions that describe various regulatory conditions on the cost $f$.

\begin{assump}\longthmtitle{Strong convexity condition}\label{asm:str_convexity}
{\rm 
The cost function $f(\vect{x},t):\real^n\times\real_{\geq 0}\rightarrow\real$ is twice continuously differentiable with respect to $\vect{x}$, and is $m$-strongly convex in $\vect{x}$. That is, at each $t\in\real_{\geq0}$ and $
  \vect{x},\vect{z}\in \!\real^n,~\vect{x}\!\neq\!\vect{z}$,  $f$ satisfies 
\begin{equation*}
  m\|\vect{z}-\vect{x}\|^2 \leq (\vect{z}-\vect{x})^\top(
  \nabla_{\vect{x}} f(\vect{z},t)-\nabla_{\vect{x}} f(\vect{x},t))\qquad \boxend.
\end{equation*}
} 
\end{assump}

\begin{assump}\sloppy[Global Lipschitz gradient  ]\label{asm:M_Lip}
{\rm 
The cost function $f(\vect{x},t):\real^n\times\real_{\geq 0}\rightarrow\real$ is
is twice differentiable with respect to $\vect{x}$, and has a \emph{$M$-Lipschitz} gradient. That is, for each  $M$-strongly convex in $\vect{x}$. That is, at each $t\in\real_{\geq0}$ and for any $
  \vect{x},\vect{z}\in \!\real^n,~\vect{x}\!\neq\!\vect{z}$,  $\nabla_{\vect{x}}f$ satisfies 
\begin{equation*}
 (\vect{z}-\vect{x})\!^\top\!(
  \nabla_{\vect{x}} f(\vect{z},t)-\nabla_{\vect{x}} f(\vect{x}),t)\!\leq
  M\|\vect{z}-\vect{x}\|^2.\qquad \boxend
\end{equation*}
 }
\end{assump}
In what follows we define
   \begin{align}
    f^\star(t)=f(\vectsf{x}^\star(t),t) \quad \vect{x}\in\real^n, ~t\in\real_{\geq 0}.
\end{align}
The optimality condition~\eqref{eq::KKT_timeVar} defines the solution of the optimization problem~\eqref{eq::opt}. In the case of differentiable costs, the optimal trajectory $t\mapsto\vectsf{x}_t^\star$ should satisfy the differential equation~\eqref{eq::x_star_trajectory}. Under Assumptions~~\ref{asm:str_convexity} and \ref{asm:M_Lip}, the trajectory $t\mapsto\vectsf{x}_t^\star$ given by~\eqref{eq::x_star_trajectory} with an initial value of $\vectsf{x}^\star(0)$ is unique~\cite{ALD-RTR:09}.

To ensure convergence of our proposed algorithms, we impose additional conditions on the cost function, as stated in Assumption~\ref{asm:bound_dfstar}. These conditions, along with Assumption~\ref{asm:str_convexity}, are similar to the assumptions made in existing second-order time-varying optimization algorithms.
\begin{assump}\longthmtitle{Smoothness of the cost function}\label{asm:bound_dfstar}
{\rm 
 The function  $f(\vect{x},t):\real^n\times\real_{\geq 0}\rightarrow\real$ is  continuously differentiable and sufficiently smooth. Specifically,
there exists a bound on the derivative of $f(\vect{x},t)$ as 
\begin{align*}
& |\nabla_{t}f(\vect{x},t)|\leq K_1,~\, \|\nabla_{\vect{x}t}f(\vect{x},t)\|\leq K_2,~\, |\nabla_{tt}f(\vect{x},t)|\leq\! K_3
\end{align*}
for any $\vect{x}\in\real^n$, $t\in\real_{\geq0}$. }\boxend
\end{assump}
According to \cite[Lemma 3.3]{HKK:02}, Assumption~\ref{asm:bound_dfstar} implies that the cost function $f(\vect{x},t)$ and its first derivatives $\nabla_{\vect{x}}f(\vect{x},t)$ and $\nabla_{t}f(\vect{x},t)$ are globally Lipschitz in $t\in\real_{\geq0}$ with constants $K_1$, $K_2$ and $K_3$, respectively.

\begin{lem}\longthmtitle{Bound on the  cost difference of the optimizer}\label{bound_traj_opt}{\rm
Consider the optimization problem~\eqref{eq::opt} under Assumptions~\ref{asm:str_convexity},~\ref{asm:M_Lip} and~\ref{asm:bound_dfstar}. Then,
\begin{align}\label{eq::bound_optimal_trajectory}
    |f^\star(t_{k+1})-f^\star(t_{k})|\leq \psi,
\end{align}
where $\psi=\delta (K_1 + \frac{\delta}{2}K_3) + \frac{K^2_2 \delta^2}{2m} (\frac{M\delta}{m} + 2)$,
with $\delta=t_{k+1}-t_k\in\real_{>0}$ being the sampling timestep of the optimal cost function across time.
}

\end{lem}



\section{Objective statement}\label{sec::Prob_formu}

Algorithms with a computational complexity of $O(n
)$ should use the first-order derivatives of the cost with respect to $\vect{x}$. For a first-order solver for problem~\eqref{eq::opt}, previous work such as~\cite{AYP:05} investigated the use of the conventional gradient descent~algorithm 
\begin{align}\label{eq::grad_desc}
    &\vect{x}_{k+1}=\vect{x}_k-\alpha\, \nabla_{\vect{x}} f(\vect{x}_k,t_{k+1}).
\end{align}
Considering the first-order approximation of the cost 
\begin{equation}
\begin{aligned}
f(\vect{x}_{k+1},t_{k+1})\approx& f(\vect{x}_{k},t_{k+1}) \nonumber \\
&+\nabla_{\vect{x}}f(\vect{x}_{k},t_{k+1})^\top(\vect{x}_{k+1}-\vect{x}_{k})\nonumber,
\end{aligned}
\end{equation}
the gradient decent algorithm~\eqref{eq::grad_desc} certainly results in a reduction of the function in each $t_{k+1}$, i.e., $f(\vect{x}_{k+1},t_{k+1})\leq  f(\vect{x}_{k},t_{k+1})$ in first order approximation sense. In contrast, the first-order approximation over time from $t_k$ to $t_{k+1}$, 
\begin{align}\label{eq::taylor_time} f(\vect{x}_{k},t_{k+1})\approx& f(\vect{x}_{k},t_{k})+\nabla_{t}f(\vect{x}_{k},t_k)\,(t_{k+1}-t_k), \end{align} 
shows that the gradient descent algorithm~\eqref{eq::grad_desc} does not take into account $\nabla_{t}f(\vect{x}_{k},t_k)\,(t_{k+1}-t_k)$ and the change in cost over time. Iterative optimization algorithms for unconstrained problems aim to achieve a successive descent objective. However, if $\nabla_{t}f(\vect{x}_{k},t_k)>0$, it follows from~\eqref{eq::taylor_time} that $f(\vect{x}_{k},t_{k+1})>f(\vect{x}_{k},t_{k})$. Consequently, the function reduction at $t_{k+1}$ after applying the gradient descent algorithm~\eqref{eq::grad_desc} is not the same as when $f(\vect{x}_{k},t_{k+1})\approx f(\vect{x}_{k},t_{k})$ (in the first-order sense). Therefore, the gradient descent algorithm is expected to yield poor tracking performance during time intervals when $\nabla_{t}f(\vect{x}(t),t)>0$.

Let us re-write the gradient decent  algorithm~\eqref{eq::grad_desc} as 
\begin{subequations}\label{eq::grad_desc_alter}
\begin{align}\label{eq::grad_desc_alt}
\textsf{Prediction}:~\vect{x}^-_{k+1}&=\vect{x}_{k},\\
\textsf{~~~~Update}:~ \vect{x}_{k+1}&=\vect{x}_{k+1}^--\alpha\, \nabla_{\vect{x}} f(\vect{x}_{k+1}^-,t_{k+1}),
\end{align}
\end{subequations}
where, we use $\vect{x}^-_{k+1}$ to represent the \emph{predicted} decision variable at time $t_{k+1}$ and $\vect{x}_{k+1}$ as the \emph{updated} decision variable at $t_{k+1}$. In a time-varying cost setting, the expectation is that, at each time $t_{k+1}$, the updated decision variable gets close to $\vectsf{x}^\star(t_{k+1})$ via through function descent. \blue{Our objective in this paper is to design a carefully crafted prediction step that accounts for cost variations over time using the information from $\nabla_t f$ or $\nabla_{\vect{x}t} f$ such that $f(\vect{x}_{k+1}^-,t_{k+1})\leq f(\vect{x}_{k},t_{k+1})$.This prediction step will maintain a computational cost at most in the same order of magnitude as $\nabla_{\vect{x}} f$, $O(n)$, used in the update step. By employing a prediction rule, we aim to provide a ``warm" start for the update step so that better tracking performance is achieved than with the gradient descent algorithm~\eqref{eq::grad_desc_alter}.}

\begin{algorithm}[t]
\caption{$O(n)$ algorithm with $\nabla_t f(\vect{x}_k,t_k)$} 
\label{Alg::1}
\begin{algorithmic}[1]
\State \textbf{Initialization:} $\vect{x}_0 \in \real^n$, $\eps, \delta \in \real_{>0}$, $f(\vect{x}_0,t_0)\in\real$
\Statex \!\!\!\!\!\!\!\!\underline{\textsf{Prediction}}
\If{$\|\nabla_{\vect{x}}f(\vect{x}_k,t_k)\|\geq \eps$}
    \State $\vect{x}^-_{k+1} = \vect{x}_{k} - \delta \frac{|\nabla_t f(\vect{x}_k,t_k)|}{\|\nabla_{\vect{x}}f(\vect{x}_k,t_k)\|^2} \nabla_{\vect{x}}f(\vect{x}_k,t_k)$
\Else
    \State $\vect{x}^-_{k+1} = \vect{x}_{k}$
\EndIf
\Statex \!\!\!\!\!\!\!\!\underline{\textsf{Update}} 
\State $\vect{x}_{k+1} = \vect{x}^{-}_{k+1} - \alpha \nabla_{\vect{x}} f(\vect{x}^{-}_{k+1},t_{k+1})$
\end{algorithmic}
\end{algorithm}

\section{$O(n)$ algorithms with first-order derivatives}\label{first_order}
In this section, we introduce a set of algorithms to solve the optimization problem~\eqref{eq::opt} using only the first-order derivatives of the cost with respect to the decision variable $\vect{x}$, resulting in $O(n)$ computational complexity. These algorithms incorporate prediction and update steps, leveraging temporal variations in the cost to enhance convergence performance.

\begin{figure}[t]
  \centering
    \includegraphics[scale=0.29]{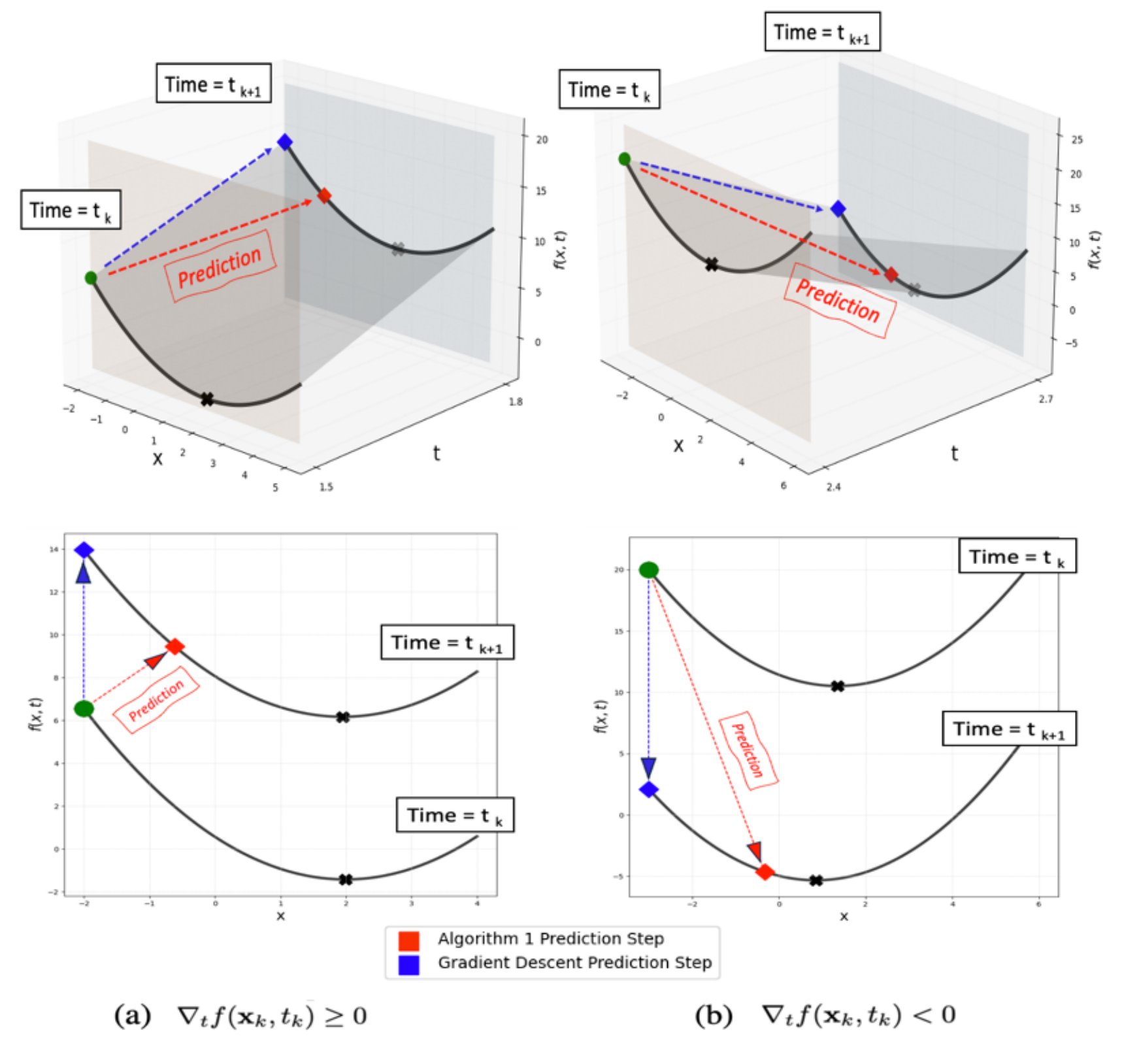}
      \caption{\small{An example case that demonstrates the role of the prediction step (line 3) of Algorithm \ref{Alg::1}. As we can see in this example, for both cases of $\nabla_t f(\vect{x}_k,t_k) \geq 0 $ (plots in the left column) and $\nabla_t f(\vect{x}_k,t_k) < 0 $ (plots in the right column)  $ f^-_1(t_{k+1})\leq f^-_{\text{g}}(t_{k+1})$.}}
      \label{fig:example-1}
\end{figure}
Our first proposed algorithm, Algorithm~\ref{Alg::1}, employs a prediction step (line 3) that involves adjusting the local state based on the rate of change of the cost function over time. The update step (line 7) is a gradient descent step at a fixed time $t_{k+1}$. \blue{In our preliminary work~\cite[Lemma IV.1]{MR-SSK:24}, for any function $f$ differentiable in $\vect{x}$ and $t$, we showed that the prediction step results in 
$ f^-_1(t_{k+1})\leq f^-_g(t_{k+1})$, where 
$f^-_g(t_{k+1})$ and $f^-_1(t_{k+1})$ are the function value computed in the first-order sense from~
\begin{align}\label{eq::first_order_approx}
f(\vect{x}^-_{k+1},t_{k+1})\approx &f(\vect{x}_k,t_{k})\,+\nabla_{\vect{x}} f(\vect{x}_k,t_k)(\vect{x}^-_{k+1}-\vect{x}_k)\nonumber\\
&+\nabla_t f(\vect{x}_k,t_k)(t_{k+1}-t_k).
\end{align}
where $\vect{x}^-_{k+1}$ generated respectively by the gradient descent algorithm~\eqref{eq::grad_desc_alter} and Algorithm~\ref{Alg::1}. This result is illustrated in Fig.~\ref{fig:example-1}.} A more accurate comparison between gradient descent and Algorithm \ref{Alg::1} is obtained when considering the ranges of $\delta$ for which the first-order approximation is the dominant factor in how the function changes across $\vect{x}$ and $t$. This comparison is further explained in the next result.

\begin{lem}\longthmtitle{Function reduction due to Algorithm~\ref{Alg::1}'s prediction step}\label{lemma::valid_delta}
{\rm Consider the gradient descent algorithm~\eqref{eq::grad_desc_alter} and Algorithm~\ref{Alg::1} \blue{under Assumptions~\ref{asm:M_Lip} and \ref{asm:bound_dfstar}}. Let $\delta=(t_{k+1}-t_k)$ be the same for both algorithms. Let $f^-_g(t_{k+1})=f(\vect{x}^-_{k+1},t_{k+1})$ and $f^-_1(t_{k+1})=f(\vect{x}^-_{k+1},t_{k+1})$ be, respectively, the function value of the gradient descent algorithm and Algorithm~\ref{Alg::1} after their respective prediction step. Suppose $\vect{x}_k$ for both algorithms has the same value. Then, for any $\eps\in\real_{\geq0}$, we have \mbox{$ f^-_1(t_{k+1})\leq f^-_g(t_{k+1})$} provided that \mbox{$0 < \delta \leq \frac{2\epsilon^2}{K_1 M + 2\epsilon\,K_2}$}.}
\end{lem}
\begin{proof}
    Taylor series expansion of $f^-_\text{g}(\vect{x}_{k+1}^-,t_{k+1})$ according to~\eqref{eq::exact_taylor} and 
    substituting for $\vect{x}^{-}_{k+1}$ from the prediction step~\eqref{eq::grad_desc_alt} of the gradient descent algorithm gives
    \begin{equation}\label{eq:grad_des_pred}
      \begin{aligned}
\!\! \!f^-_\text{g}(t_{k+1})\!=\! f(\vect{x}_{k},t_{k})\! +\!  \delta\nabla_{t} f(\vect{x}_{k},t_k)\!+\!\frac{\delta^2}{2}\nabla_{tt}\! f(\vect{\vect{\zeta}},\theta).
\end{aligned}
    \end{equation}
    If $\|\nabla_{\vect{x}}f(\vect{x}_k,t_k)\|\leq\eps$, the prediction steps of Algorithm~\ref{Alg::1} and the gradient descent algorithm are the same, i.e., $f^-_1(t_{k+1})=f^-_g(t_{k+1})$. If $\|\nabla_{\vect{x}}f(\vect{x}_k,t_k)\|\geq\eps$, substituting the prediction step of Algorithm \ref{Alg::1}  (line $3$) in Taylor series expansion~\eqref{eq::exact_taylor} of $f^-_1(\vect{x}_{k+1}^-,t_{k+1})$ results in
\begin{align}\label{eq:f_pred}
 &f^-_1(t_{k+1})= \underbrace{f(\vect{x}_{k},t_{k})+\delta\,\nabla_t f(\vect{x}_k,t_k)+\frac{\delta^2}{2}\nabla_{tt} f(\vect{\vect{\zeta}},\theta)}_{f^-_\text{g}(t_{k+1})}-\nonumber\\
 &\delta|\nabla_t f(\vect{x}_k,t_k)|-\frac{\delta^2|\nabla_{t} f(\vect{x}_{k},t_k)|}{\| \nabla_{\vect{x}} f(\vect{x}_{k},t_k)\|^2}  \nabla_{\vect{x}t}f (\vect{\vect{\zeta}},\theta)^\top \nabla_{\vect{x}} f(\vect{x}_{k},t_k)\nonumber \\
 &+\frac{\delta^2\nabla_t f(\vect{x}_{k},t_k)^2}{2\|\nabla_{\vect{x}} f(\vect{x}_{k},t_k)\|^4}
\nabla_{\vect{x}} f(\vect{x}_{k},t_k)^\top \nabla_{\vect{xx}}f(\vect{\vect{\zeta}},\theta)\nabla_{\vect{x}} f(\vect{x}_{k},t_k).
\end{align}
Subsequently, $\frac{1}{\|\nabla_{\vect{x}}f(\vect{x}_k,t_k)\|}\leq\frac{1}{\eps}$ along with Assumptions~\ref{asm:M_Lip},~\ref{asm:bound_dfstar}   
lead to 
\begin{align*}
 &f^-_1(t_{k+1})- f^-_\text{g}(t_{k+1})\leq  -\delta\,|\nabla_t f(\vect{x}_k,t_k)|\,+\nonumber \\
 &\qquad\frac{\delta^2|\nabla_{t} f(\vect{x}_{k},t_k)| K_2}{\epsilon}+\frac{\delta^2 |\nabla_t f(\vect{x}_k,t_k)|K_1\,M}{2\eps^2}.
\end{align*}
Then, $f^-_1(t_{k+1}) -f^-_\text{g}(t_{k+1})\leq 0$ for the stepsize range given in the statement, completing the proof.
\end{proof}

\blue{The fact that the prediction step in Algorithm~\ref{Alg::1} results in $ f^-_1(t_{k+1})\leq f^-_{\text{g}}(t_{k+1})$ is of pivotal importance, as it demonstrates the impact of utilizing the temporal information $\nabla_t f$ of the cost function in achieving a better starting point in the next time step for the gradient descent update. During the update step, the algorithm takes a descent step towards the minimizer $\vectsf{x}^\star_{k+1}$. According to standard results for nonlinear optimization with a constant step-size, this will reduce the function value if $0<\alpha\leq \frac{1}{M}$~\cite{DPB:99}. The more update steps taken, the better the descent achieved. However, exact convergence to $\vectsf{x}^\star_{k+1}$ during the update step is asymptotic and cannot be accomplished in the finite time $t_{k+2}-t_{k+1}=\delta$, as we must transition to the cost at time $t_{k+2}$.

In the following, we examine the convergence behavior of Algorithm~\ref{Alg::1} over time, assuming only one update step is taken at each $t_{k+1}$. This analysis aims to demonstrate that under bounded temporal variations, we can expect convergence to a neighborhood of the optimal trajectory with some steady-state error. }
 In the following, we let $f_1(t_k)=f(\vect{x}_{k},t_k),$ where $\vect{x}_k$ is the updated decision variable of Algorithm~\ref{Alg::1} at timestep $k$. 

\begin{thm}\longthmtitle{Convergence guarantee of Algorithm~\ref{Alg::1}}\label{thm::main1}
{\rm 
Consider Algorithm~\ref{Alg::1} under Assumptions~\ref{asm:str_convexity},~\ref{asm:M_Lip} and \ref{asm:bound_dfstar}.  Then,  Algorithm~\ref{Alg::1} converges to the neighborhood of the optimum solution of~\eqref{eq::opt} as characterized by 
\begin{align}\label{eq::bound_on_solu_1}
f_1(t_{k+1})-f^{\star}(t_{k+1})\leq \,&(1-2\kappa\alpha m)^k\,(f_1(t_{0})-f^{\star}(t_{0}))\nonumber\\
&+\big(1-(1-2\kappa\alpha m)^k\big)\, E_1,
\end{align}
where $E_1=\frac{\psi}{4\kappa^2\alpha^2 m} +\frac{\max(\gamma\,\delta^2,\mu\,\delta)}{4\kappa^2\alpha^2 m^2}$, $\psi = \delta (K_1 + \frac{\delta}{2}K_3) + \frac{K^2_2 \delta^2}{2m} (\frac{M\delta}{m} + 2)$, $\delta=t_{k+1}-t_k$, $\gamma= \frac{2K_1}{\delta} + \frac{M}{2\epsilon^2}K_1^2+\frac{1}{2}K_3 + \frac{1}{\epsilon}K_1K_2$, $\mu=K_1+\frac{\delta}{2}K_3$ and $\kappa=(1-\alpha M/2)$, provided that $0<\alpha\leq \frac{1}{2M}$.
}
\end{thm}
See appendix for the proof of Theorem~\ref{thm::main1}.
\smallskip
\begin{rem}\longthmtitle{Ultimate tracking bound of Algorithm 1}\label{rem::Algorithm1_1}
{
\rm
The tracking bound of established for Algorithm~1 in~\eqref{eq::bound_on_solu_1} shows that as $k\to\infty$, $(1-2\kappa\alpha m)^k\to0$ \blue{with a geometric convergence rate  determined by $\kappa$}. Thus, the effect of initialization error $f_1(t_0)-f^{\star}(t_{0})$ vanishes with time. Moreover, the ultimate bound on $f_1(t_{k+1})-f^{\star}(t_{k+1})$ as $k\to\infty$ is $E_1=\frac{\psi}{4\kappa^2\alpha^2 m }+\frac{\max(\gamma\delta^2,\mu\delta)}{4\kappa^2\alpha^2 m^2}$. \blue{As expected, the ultimate tracking bound depends on the magnitude of temporal variations in the cost. If these variations are zero (i.e., $K_1=K_2=K_3=0$), the tracking error vanishes.} Note that the optimal value of $\eps$ corresponding to the lowest bound in~\eqref{eq::bound_on_solu_1} is obtained as the solution of  $\gamma\delta^2=\mu\delta$,
which can be calculated numerically.
}
\boxend
\end{rem}

\smallskip\begin{rem}\longthmtitle{Improved upper bounds for small $\delta$}
{\rm
Note that, the upper bound in Theorem \ref{thm::main1} is derived for any arbitrary $\delta$. Examining the proof of Theorem \ref{thm::main1}, we can see that the upper bound in Theorem \ref{thm::main1} can be written more precisely when we distinguish the function changes in time instances $\mathcal{T} = \left\{k \, \middle| \, \|\nabla_{\vect{x}} f(\vect{x}_k, t_k)\| \geq \epsilon \right\}$ that the prediction step is active and when this step is not active, arriving at   
\begin{align*}
&f_1(t_{k+1})-f^{\star}(t_{k+1})\leq \,(1-2\kappa\alpha m)^k\,(f_1(t_{0})-f^{\star}(t_{0}))\nonumber\\
&+ \big(1-(1-2\kappa\alpha m)^k\big) \frac{\psi}{4\kappa^2\alpha^2 m} \,+\nonumber \\
&\frac{1}{4\kappa^2\alpha^2 m}(\gamma \delta^2 \!\sum_{i \in \mathcal{T}} (1-2\kappa\alpha m)^{k-i} \!+ \!\mu \delta \!\sum_{i \notin \mathcal{T}} (1-2\kappa\alpha m)^{k-i}).
\end{align*}
Note that for small enough $\delta$, $\gamma \delta^2 \leq \mu \delta$. Consequently, the upper bound improves over the gradient descent algorithm by $\frac{(\mu \delta - \gamma \delta^2)}{4\kappa^2\alpha^2 m} \sum_{i \in \mathcal{T}} (1-2\kappa\alpha m)^{k-i}$ for small $\delta$, as expected from the result in Lemma \ref{lemma::valid_delta}.\boxend}
\end{rem}

\subsection{An $O(n)$ algorithm without explicit knowledge of variation of cost across time}
Algorithm~\ref{Alg::1} leverages explicit knowledge of the temporal variation in cost, $\nabla_{t}f(\vect{x}_{k},t_k)$. However, in some applications, especially where costs are derived from streaming data, this information may not be available despite temporal relationships in cost values over time. For such scenarios, we propose Algorithm~\ref{Alg::2}, which retains the prediction-correction structure of Algorithm~\ref{Alg::1} but uses an approximation for $\nabla_t f(\vect{x}_k$, which is 
\begin{align}
 \nabla_{t} f(\vect{x}_{k},t_k)\approx\frac{f(\vect{x}_{k},t_k)-f(\vect{x}_{k},t_{k-1})}{t_k-t_{k-1}}.
\end{align}
At the expense of higher computational and storage costs, higher-order differences can be used to construct a better approximation of $\nabla_t f$. Moreover, for continuous-time  algorithms, \cite{marchi2024framework} presents a continuous-time derivative estimation scheme based on ``dirty-derivatives".

\smallskip
\begin{algorithm}[t]
\caption{$O(n)$ algorithm with $\nabla_t f(\vect{x}_k,t_k)$ approximation}
\begin{algorithmic}[1]
\State \textbf{Initialization:} $\vect{x}_0\in\real^n$, $\eps\in\real_{>0}$, $f(\vect{x}_0,t_0)\in\real$
\Statex \!\!\!\!\!\!\!\!\underline{\textsf{Prediction}}
\If{$\|\nabla_{\vect{x}}f(\vect{x}_k,t_k)\|\geq \eps$}
    \State $\vect{x}^-_{k+1} = \vect{x}_{k} - \frac{|f(\vect{x}_k,t_{k}) - f(\vect{x}_k,t_{k-1})|}{\|\nabla_{\vect{x}}f(\vect{x}_k,t_k)\|^2} \nabla_{\vect{x}}f(\vect{x}_k,t_k)$
\Else
    \State $\vect{x}^-_{k+1} = \vect{x}_{k}$
\EndIf
\Statex \!\!\!\!\!\!\!\!\underline{\textsf{Update}}
\State $\vect{x}_{k+1} = \vect{x}^{-}_{k+1} - \alpha \nabla_{\vect{x}}f(\vect{x}^{-}_{k+1}, t_{k+1})$
\end{algorithmic}
\label{Alg::2}
\end{algorithm}

\begin{rem}\longthmtitle{Prediction step of Algorithm~\ref{Alg::2} results in better function reduction than the gradient descent algorithm}\label{delta:Algorithm2}
{
\rm
Starting at same $\vect{x}_k$, an argument similar to Lemma~\ref{lemma::valid_delta} can be made about $f^-_2(t_{k+1})\leq f^-_g(t_{k+1})$ provided $0 < \delta \leq  \frac{2\epsilon^2}{(K_1 + \frac{\delta}{2}K_3) M + 2\epsilon\,K_2}$. To establish this range for $\delta $ we used $\frac{|f(\vect{x}_{k},t_k)-f(\vect{x}_{k},t_{k-1})|}{t_k-t_{k-1}} \leq K_1 + \frac{\delta}{2}K_3$.\boxend
}
\end{rem}


\vspace{0.1in}

\begin{thm}
\longthmtitle{Convergence guarantee of Algorithm~\ref{Alg::2}}\label{thm::main2}
{\rm 
Consider Algorithm~\ref{Alg::1} under Assumptions~\ref{asm:str_convexity},~\ref{asm:M_Lip} and \ref{asm:bound_dfstar}. Then, the Algorithm~\ref{Alg::2} converges to the neighborhood of the optimum solution of~\eqref{eq::opt} as characterized by 
\begin{align}\label{eq::bound_on_solu_2}
f_2(t_{k+1})-f^{\star}(t_{k+1}) \leq\,&
(1-2\kappa\alpha m)^k\,(f_2(t_{0})-f^{\star}(t_{0}))\nonumber\\
&+\big(1-(1-2\kappa\alpha m)^k\big)\, E_2,
\end{align}
where $E_2=\frac{\psi}{4\kappa^2\alpha^2 m}+\frac{\max(\gamma'\delta^2,
\mu\delta)}{4\kappa^2\alpha^2 m^2}$, $\psi = \delta (K_1 + \frac{\delta}{2}K_3) + \frac{K^2_2 \delta^2}{2m} (\frac{M\delta}{m} + 2)$, $\gamma'=K_3 + \frac{2K_1}{\delta} + \frac{K_1^2 M}{\epsilon^2} + \frac{K_2(K_1 + \frac{\delta}{2} K_3)}{\epsilon} + \frac{\delta^2 K_3^2 M}{4\epsilon^2}$, 
$\mu=K_1+\frac{\delta}{2}K_3$ and 
$\kappa=(1-\alpha M/2)$,
provided that $0<\alpha\leq \frac{1}{2M}$.
}
\end{thm}
See the appendix for proof of Theorem~\ref{thm::main2}.

Based on Theorem~\ref{thm::main2}, a comparable statement to Remark~\ref{rem::Algorithm1_1} can be made regarding the ultimate tracking bound of Algorithm~\ref{Alg::2}. It is also noteworthy that the tracking error of Algorithm~\ref{Alg::2}, given its use of an estimate for $\nabla_{t} f(\vect{x}_{k},t_k)$, may be larger than that of Algorithm~\ref{Alg::1} because $\gamma' \geq \gamma$.

\begin{algorithm}[t]
\caption{$O(n)$ algorithm with $\nabla_t f(\vect{x}_k,t_k)$ and $\nabla_{\vect{x}t} f(\vect{x}_k,t_k)$}
\begin{algorithmic}[1]
\State \textbf{Initialization:} $\vect{x}_0\in\real^n$, $\eps, \delta\in\real_{>0}$, $f(\vect{x}_0,t_0)\in\real$
\Statex \!\!\!\!\!\!\!\!\underline{\textsf{Prediction}}
\If {$\|\nabla_{\vect{x}}f(\vect{x}_k,t_k) + \delta\nabla_{\vect{x}t} f(\vect{x}_k,t_k)\|\geq \eps$ and $\nabla_{\vect{x}t} f(\vect{x}_k,t_k)^\top \nabla_{\vect{x}} f(\vect{x}_k,t_k) \leq 0$}
    \State \mbox{$\vect{x}^-_{k+1}=
    \vect{x}_{k} -\frac{\delta|\nabla_t f(\vect{x}_k,t_{k})|}{\|\nabla_{\vect{x}}f(\vect{x}_k,t_k) + \delta \nabla_{\vect{x}t} f(\vect{x}_k,t_k) \|^2}$} 
    \par
    \mbox{$\big(\nabla_{\vect{x}}f(\vect{x}_k,t_k) + \delta\nabla_{\vect{x}t} f(\vect{x}_k,t_k)\big)$}
\ElsIf{$\|\nabla_{\vect{x}}f(\vect{x}_k,t_k)\|\geq \eps$}
    \State $\vect{x}^-_{k+1} = \vect{x}_{k} - \delta \frac{|\nabla_t f(\vect{x}_k,t_{k})|}{\|\nabla_{\vect{x}}f(\vect{x}_k,t_k)\|^2} \nabla_{\vect{x}}f(\vect{x}_k,t_k)$
\Else
    \State $\vect{x}^-_{k+1} = \vect{x}_{k}$
\EndIf
\Statex \!\!\!\!\!\!\!\!\underline{\textsf{Update}}
\State $\vect{x}_{k+1} = \vect{x}^{-}_{k+1} - \alpha \nabla_{\vect{x}} f(\vect{x}^{-}_{k+1}, t_{k+1})$
\end{algorithmic}
\label{Alg::3}
\end{algorithm}

\begin{figure}[t]
  \centering
    \includegraphics[scale=0.28]{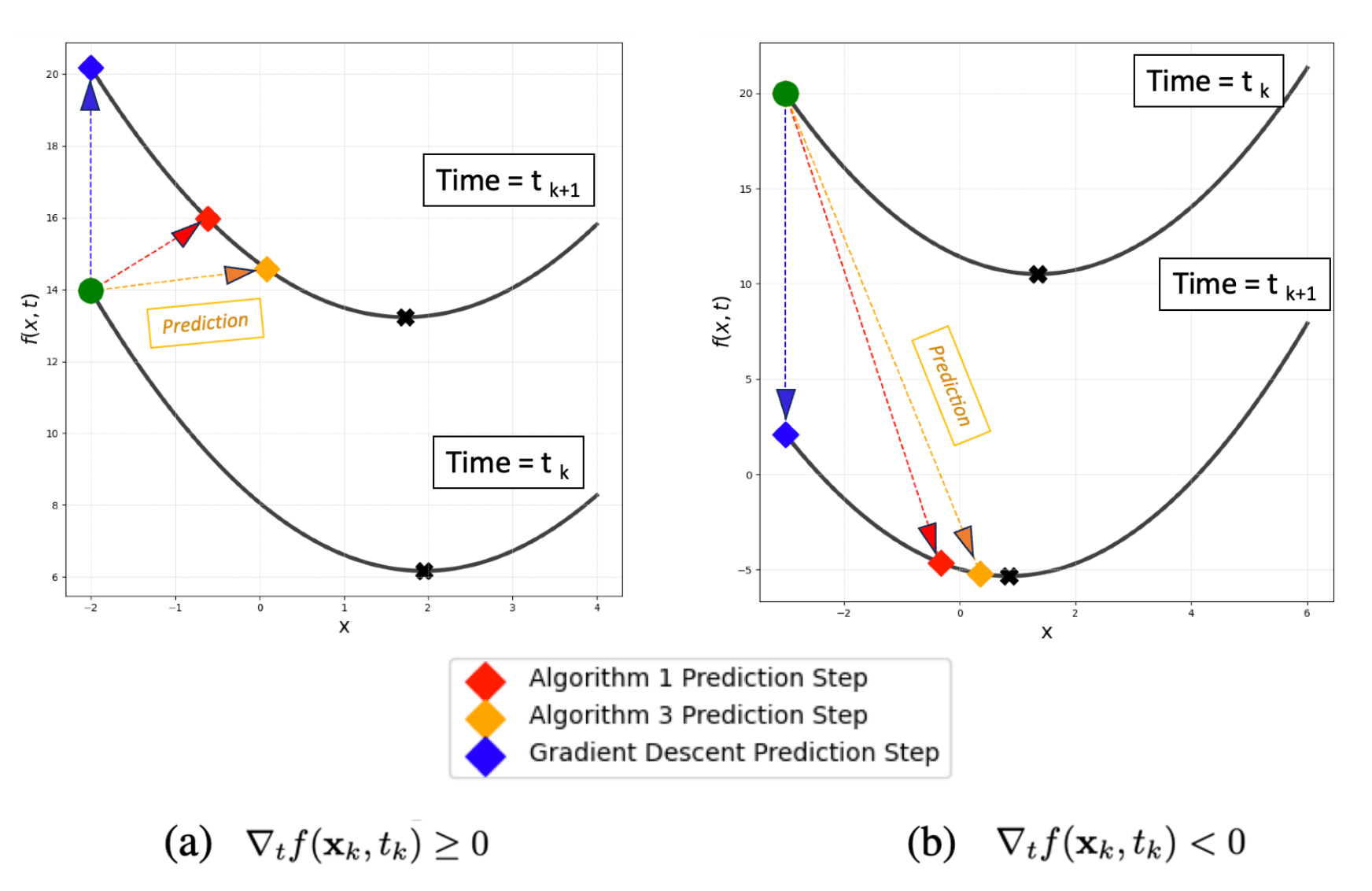}
      \caption{\small{An example case that demonstrates the role of the prediction step (line 3) of Algorithm \ref{Alg::3}. As we can see in this example, for both cases of $\nabla_t f(\vect{x}_k,t_k) \geq 0 $ and $\nabla_t f(\vect{x}_k,t_k) < 0 $ the statement of Lemma~\ref{lem::Alg3_vs_Alg1} holds, i.e., $ f^-_3(t_{k+1})\leq f^-_1(t_{k+1})\leq f^-_{\text{g}}(t_{k+1})$.}}
      \label{fig:example-3}
\end{figure}

\subsection{An $O(n)$ algorithm that uses the time derivative of the gradient}

So far, we have noted the improvement brought by using temporal information $\nabla_t f$ in the prediction step. Now, we explore whether incorporating additional temporal insights, such as changes in the gradient over time, can further amplify this advantage. To achieve this, we introduce Algorithm~\ref{Alg::3}, which utilizes extra temporal information about the cost function, specifically $\nabla_{\vect{x}t} f(\vect{x}_k,t_k)$ that still needs an $O(n)$ computational cost. The result below formally shows that, the enhanced prediction step of Algorithm~\ref{Alg::3} leads to improved performance; see Fig.~\ref{fig:example-3} for a demonstration. \blue{The design of the prediction step in Algorithm~3 is derived from the first-order Taylor series expansion of $f(\vect{x}_{k+1}^-, t_{k+1})$. It purposefully crafts a prediction step using only $\nabla_t f(x)$, $\nabla_{\vect{x}} f$, and $\nabla_{\vect{x}t} f$ to ensure $f^-_3(t_{k+1}) \leq f^-_1(t_{k+1})$, as detailed in the proof of the result below.}

\begin{lem}\longthmtitle{Prediction step of Algorithm~\ref{Alg::3} out performs that of Algorithm~\ref{Alg::1}}\label{lem::Alg3_vs_Alg1}
{\rm Consider Algorithm~\ref{Alg::1} and Algorithm~\ref{Alg::3}. Let $\delta=(t_{k+1}-t_k)$ be the same for both algorithms. Let $f^-_1(t_{k+1})=f(\vect{x}^-_{k+1},t_{k+1})$ and $f^-_3(t_{k+1})=f(\vect{x}^-_{k+1},t_{k+1})$ be computed in the first-order sense from~\eqref{eq::first_order_approx}, respectively, for Algorithm~\ref{Alg::1} and Algorithm~\ref{Alg::3} after their respective prediction step. Suppose that $\vect{x}_k$ for both algorithms has the same value. Then, for any $\eps\in\real_{\geq0}$, we have $ f^-_3(t_{k+1})\leq f^-_1(t_{k+1})$. }
\end{lem}

\begin{proof}
Considering the first-order Taylor series expansion for Algorithm \ref{Alg::1}, we arrive at 
\vspace{-0.08in}
\begin{align}\label{eq::pred_alg1_first_order}
f^-_\text{1}(t_{k+1})
\!\approx\,&f(\vect{x}_k,t_{k})\!+\delta\,(\nabla_t f(\vect{x}_k,t_k)\!-\mathcal{I}|\nabla_t f(\vect{x}_k,t_k)|), 
\end{align}
where 
$\|\nabla_{\vect{x}}f(\vect{x}_k,t_k)\|\geq \eps$ leads to  $\mathcal{I}=1$, otherwise   $\mathcal{I}=0$. 

Algorithm~\ref{Alg::3}'s prediction is different than Algorithm~\ref{Alg::1}'s when condition in it's line 2 is active. Under this condition,  adding and subtracting $\delta \nabla_{\vect{x}t}  f(\vect{x}_k,t_k))^\top
(\vect{x}^-_{k+1}-\vect{x}_k)$ to the right hand side of~\eqref{eq::first_order_approx}, and substituting for $\vect{x}^-_{k+1}$ from line 3 of Algorithm~\ref{Alg::3} leads to 
\begin{align}\label{eq::pred_alg3}
    &f^-_\text{3}(t_{k+1})\approx f(\vect{x}_k,t_{k})+\delta\,(\nabla_t f(\vect{x}_k,t_k)-|\nabla_t f(\vect{x}_k,t_k)|)\,+ \nonumber \\
    & \frac{\delta\,|(\nabla_t f(\vect{x}_k,t_k)|}{\|\nabla_{\vect{x}}f(\vect{x}_k,t_k) \!+\! \delta\nabla_{\vect{x}t}  f(\vect{x}_k,t_k)\|^2} \nabla_{\vect{x}t}  f(\vect{x}_k,t_k)^\top \nabla_{\vect{x}}  f(\vect{x}_k,t_k) \nonumber\\
&-\frac{\delta\,|(\nabla_t f(\vect{x}_k,t_k)|}{\|\nabla_{\vect{x}}f(\vect{x}_k,t_k) + \delta\nabla_{\vect{x}t}  f(\vect{x}_k,t_k)\|^2} \|\nabla_{\vect{x}t}  f(\vect{x}_k,t_k)\|^2.
\end{align}
In this relationship, since $\nabla_{\vect{x}t}  f(\vect{x}_k,t_k)^\top \nabla_{\vect{x}}  f(\vect{x}_k,t_k) \leq 0$, the last two terms are always negative. Therefore, considering~\eqref{eq::pred_alg1_first_order}, regardless of the value of $\|\nabla_{\vect{x}} f(\vect{x}_k,t_k)\|$, when condition in line $2$ of Algorithm~\ref{Alg::3}'s is active, we have $f^-_3(t_{k+1})\leq f^-_1(t_{k+1})$. When this condition is not satisfied prediction step of Algorithms~\ref{Alg::1} and \ref{Alg::3} are the same and as such $f^-_3(t_{k+1})= f^-_1(t_{k+1})$, establishing that for any value of $\vect{x}_k$, $f^-_3(t_{k+1})= f^-_1(t_{k+1})$. 
\end{proof}

 A result similar to Lemma~\ref{lemma::valid_delta} can be established for ranges of stepsize $\delta$ for which the first-order approximation is the dominant component and $f^-_3(t_{k+1})\leq f^-_1(t_{k+1})$ in the exact sense. Details are omitted for the sake of brevity. Next, we present the convergence analysis of Algorithm~\ref{Alg::3}.

\begin{thm}\longthmtitle{Convergence guarantee of Algorithm~\ref{Alg::3}}\label{thm::main1::alg3}
{\rm Let 
Assumptions~\ref{asm:str_convexity} and~\ref{asm:bound_dfstar}  hold. Then,  Algorithm~\ref{Alg::3} converges to the neighborhood of the optimum solution of~\eqref{eq::opt} as characterized by
\begin{align}\label{eq::bound_on_solu_5}
f_3(t_{k+1})-f^{\star}(t_{k+1}) \leq 
&(1-2\kappa\alpha m)^k(f_3(t_{0})-f^{\star}(t_{0})) \nonumber \\
+ \big(1-(1-2\kappa\alpha m)^k\big) E_3,
\end{align}

where $E_3=\frac{\psi}{4\kappa^2\alpha^2 m}+\frac{\max(\gamma\delta^2,
\mu\delta, \eta \delta^2)}{4\kappa^2\alpha^2 m^2}$, $\eta=\frac{2}{\delta}K_1 + \frac{1}{2}K_3 + \frac{2}{\eps}K_1K_2+\frac{M}{2\eps^2}K_1^2$, $\mu=K_1+\frac{\delta^2}{2}K_3$, $\gamma= \frac{2}{\delta}K_1 + \frac{M}{2\epsilon^2}K_1^2+\frac{1}{2}K_3 + \frac{1}{\epsilon}K_1K_2 $ and $\kappa=(1-\alpha M/2)$, provided that $0<\alpha\leq \frac{1}{2M}$.
}
\end{thm}
See the appendix for proof of Theorem~\ref{thm::main1::alg3}.

Note that a similar approximate version of Algorithm \ref{Alg::3} where the exact value of $\nabla_t  f(\vect{x}_k,t_k)$ and $\nabla_{\vect{x}t}  f(\vect{x}_k,t_k)$ is not available can be obtained by the same approximation approach as in Algorithm \ref{Alg::2}. To maintain brevity, we do not delve into further details here.

 \begin{rem}
\longthmtitle{Incorporating higher order derivative of time in Algorithm~\ref{Alg::3}}\label{remark::alg3}
{
\rm
Note that Algorithm~\ref{Alg::3} provides an insight into how the changes of the gradient over time can enhance performance.Thus, by using the same amount of computing power, we can include higher order derivatives of time, e.g., $\nabla_{tt} f(\vect{x}_k,t_k)$,  to improve the performance of Algorithm~\ref{Alg::3}.
}
\boxend
\end{rem}

\section{A hybrid first- and second-order algorithms}\label{second_order}
In the proposed algorithms so far, when the trajectories of the decision variable reaches a point $\vect{x}_k$ that results in $\|\nabla_{\vect{x}} f(\vect{x}_k,t_k)\|\leq \eps$, the prediction step cannot be implemented, because of the possibility of getting too close to division by zero. Subsequently, these $O(n)$-algorithms when $\|\nabla_{\vect{x}} f(\vect{x}_k,t_k)\|\leq \eps$  behave like the gradient descent algorithm. In this section, we propose Algorithm~\ref{Alg::4}, a hybrid algorithm that switches to a second-order gradient tracking prediction step when $\|\nabla_{\vect{x}} f(\vect{x}_k,t_k)\|\leq \eps$. The expectation is that using an additional computational cost, which is only used when $\|\nabla_{\vect{x}} f(\vect{x}_k,t_k)\|\leq \eps$, we achieve better tracking. A first-order analysis of the prediction and update steps of the hybrid Algorithm~\ref{Alg::4} shines light on usefulness of using a hybrid algorithm. More specifically, when we consider the first-order approximate Taylor series expansion of the gradient after the second-order prediction (line 5 of Algorithm~\ref{Alg::4}), we arrive at 
\begin{align}\label{eq::Alg_4_pred_2nd}
    \nabla_{\vect{x}} f(\vect{x}_{k+1}^-,t_{k+1}) \approx &\,\nabla_{\vect{x}} f(\vect{x}_{k},t_k) + \nabla_{\vect{x}t} f(\vect{x}_{k},t_k) \,\delta\,+\, \nonumber\\
    & \,\nabla_{\vect{xx}} f(\vect{x}_{k},t_k)(\vect{x}^-_{k+1}\!-\!\vect{x}_k)\!\nonumber\\
    =&\,\nabla_{\vect{x}} f(\vect{x}_{k},t_k)
\end{align}
Subsequently, the gradient after the update step (line 7 of Algorithm~\ref{Alg::4}) using the first-order Taylor series explanation and substituting for $\nabla_{\vect{x}} f(\vect{x}_{k+1}^-,t_{k+1})$ from~\eqref{eq::Alg_4_pred_2nd} reads as 
\begin{align*}
   \nabla_{\vect{x}} f(\vect{x}_{k+1},t_{k+1}) \approx&\, \nabla_{\vect{x}} f(\vect{x}_{k+1}^-,t_{k+1})\,+ \nonumber \\
    & \nabla_{\vect{xx}} f(\vect{x}_{k+1}^-,t_{k+1})(\vect{x}_{k+1}-\vect{x}^-_{k+1}) \\=
    &(I-\alpha \nabla_{\vect{xx}} f(\vect{x}_{k+1}^-,t_{k+1})) \nabla_{\vect{x}} f(\vect{x}_{k},t_k).
\end{align*}
Then, for $\alpha\leq \frac{2}{m}$, in first-order sense we have  \begin{align}\label{eq:Alg_4_nabla_eps_zone}\|\nabla_{\vect{x}} f(\vect{x}_{k+1},t_{k+1})\|\leq \|\nabla_{\vect{x}} f(\vect{x}_{k},t_{k})\|\leq \epsilon.\end{align}
Since~\eqref{eq:Alg_4_nabla_eps_zone} is drawn from a first-order approximate analysis, it only gives a qualitative picture of the performance of Algorithm~\ref{Alg::4}. In fact, it is shown in~\cite{AS-AM-AK-GL-AR:16} that even if the second-order prediction step is used at all times, the algorithm can only achieve tracking with a non-zero steady-state tracking bound. Therefore, we expect that the use of a second-order prediction step in $\|\nabla_{\vect{x}} f(\vect{x}_k,t_k)\|\leq \eps$ will not result in zero steady-state tracking. Nevertheless, our numerical examples confirm the expectation that a better tracking behavior can be achieved by Algorithm~\ref{Alg::4}. The next result characterizes the convergence behavior of Algorithm~\ref{Alg::4}. 
\medskip

\begin{algorithm}[t]
\caption{Hybrid algorithm with $\nabla_t f(\vect{x}_k,t_k)$ and $\nabla_{\vect{x}t} f(\vect{x}_k,t_k)$}\label{Alg::4}
\begin{algorithmic}[1]
\State \textbf{Initialization:} $\vect{x}_0 \in \real^n$, $\eps, \delta \in \real_{>0}$, $f(\vect{x}_0, t_0) \in \real$
\Statex \!\!\!\!\!\!\!\!\underline{\textsf{Prediction}}
\If{$\|\nabla_{\vect{x}}f(\vect{x}_k, t_k)\| \geq \eps$}
    \State $\vect{x}^-_{k+1} = \vect{x}_{k} - \delta \frac{|\nabla_t f(\vect{x}_k, t_k)|}{\|\nabla_{\vect{x}} f(\vect{x}_k, t_k)\|^2} \nabla_{\vect{x}} f(\vect{x}_k, t_k)$
\Else
    \State $\vect{x}^-_{k+1} = \vect{x}_{k} - \delta (\nabla_{\vect{xx}} f(\vect{x}_k, t_k))^{-1} \nabla_{\vect{x}t} f(\vect{x}_k, t_k)$
\EndIf
\Statex \!\!\!\!\!\!\!\!\underline{\textsf{Update}}
\State $\vect{x}_{k+1} = \vect{x}^{-}_{k+1} - \alpha \nabla_{\vect{x}} f(\vect{x}^{-}_{k+1}, t_{k+1})$
\end{algorithmic}
\end{algorithm}

\begin{thm}
\longthmtitle{Convergence analysis of Algorithms~\ref{Alg::4}}\label{thm6}
{\rm Let 
Assumptions~\ref{asm:str_convexity} and~\ref{asm:bound_dfstar}  hold. Then,  Algorithm~\ref{Alg::4} converges to the neighborhood of the optimum solution of~\eqref{eq::opt} with the following upper bound
\begin{align}\label{eq::bound_on_solu_3}
\!\!\!\!f_4(t_{k+1})-f^{\star}(t_{k+1}) \!\leq 
&(1-2\kappa\alpha m)^k(f_4(t_{0})-f^\star(t_{0})) \nonumber \\
&+ \big(1-(1-2\kappa\alpha m)^k\big) E_4,
\end{align}
where $E_4=\frac{\psi}{4\kappa^2\alpha^2 m}+\frac{\max(\gamma\delta^2,
\mu\delta)}{4\kappa^2\alpha^2 m^2}$, $\psi = \delta (K_1 + \frac{\delta}{2}K_3) + \frac{K^2_2 \delta^2}{2m} (\frac{M\delta}{m} + 2)$, $\delta=t_{k+1}-t_k$, $\gamma= \frac{2K_1}{\delta} + \frac{M}{2\epsilon^2}K_1^2+\frac{1}{2}K_3 + \frac{1}{\epsilon}K_1K_2$, $\mu=\eps m K_2 + K_1 + \frac{\delta}{2}K_3 + \frac{3}{2} m \delta K_2^2$ and $\kappa=(1-\alpha M/2)$, provided that $0<\alpha\leq \frac{1}{2M}$.}
\end{thm}
See the appendix for proof of Theorem~\ref{thm6}.


Using a similar approximation approach as described in Algorithm \ref{Alg::2}, we can design a modified version of Algorithm \ref{Alg::4} for scenarios where precise values of $\nabla_t f(\vect{x}_k,t_k)$ and $\nabla_{\vect{x}t} f(\vect{x}_k,t_k)$ are not available. For the sake of brevity, we refrain from providing additional elaboration at this point.

\section{Simulation results}\label{sec::num}
In this section, we demonstrate the effectiveness of our proposed algorithms by providing results from two different examples. The first example demonstrates the results when the exact time derivative of the cost is explicitly given. However, in the second example, which involve solving a Model Predictive Control (MPC) formulated as a time-varying convex optimization problem, the time derivative of the cost is not directly available, so it must be estimated.

\subsection{A case of time-varying cost function whose time derivatives are available explicitly}
 Our first example considers a case of a time-varying cost function 
\begin{align}\label{num1_cost}
&f(x, t)= 0.5\,(x - 2\sin(t))^2 + \cos(3t) \,x,
 \end{align} 
whose time derivatives are available explicitly.
This example aims to demonstrate how the tracking behavior of the algorithms is influenced by the temporal information they use about the cost. 

The algorithms are initialized with $x_0 = 100$, and $\epsilon = 0.3$ is chosen for implementation. 
 \begin{figure}[t]
  \centering
    \includegraphics[scale=0.25]{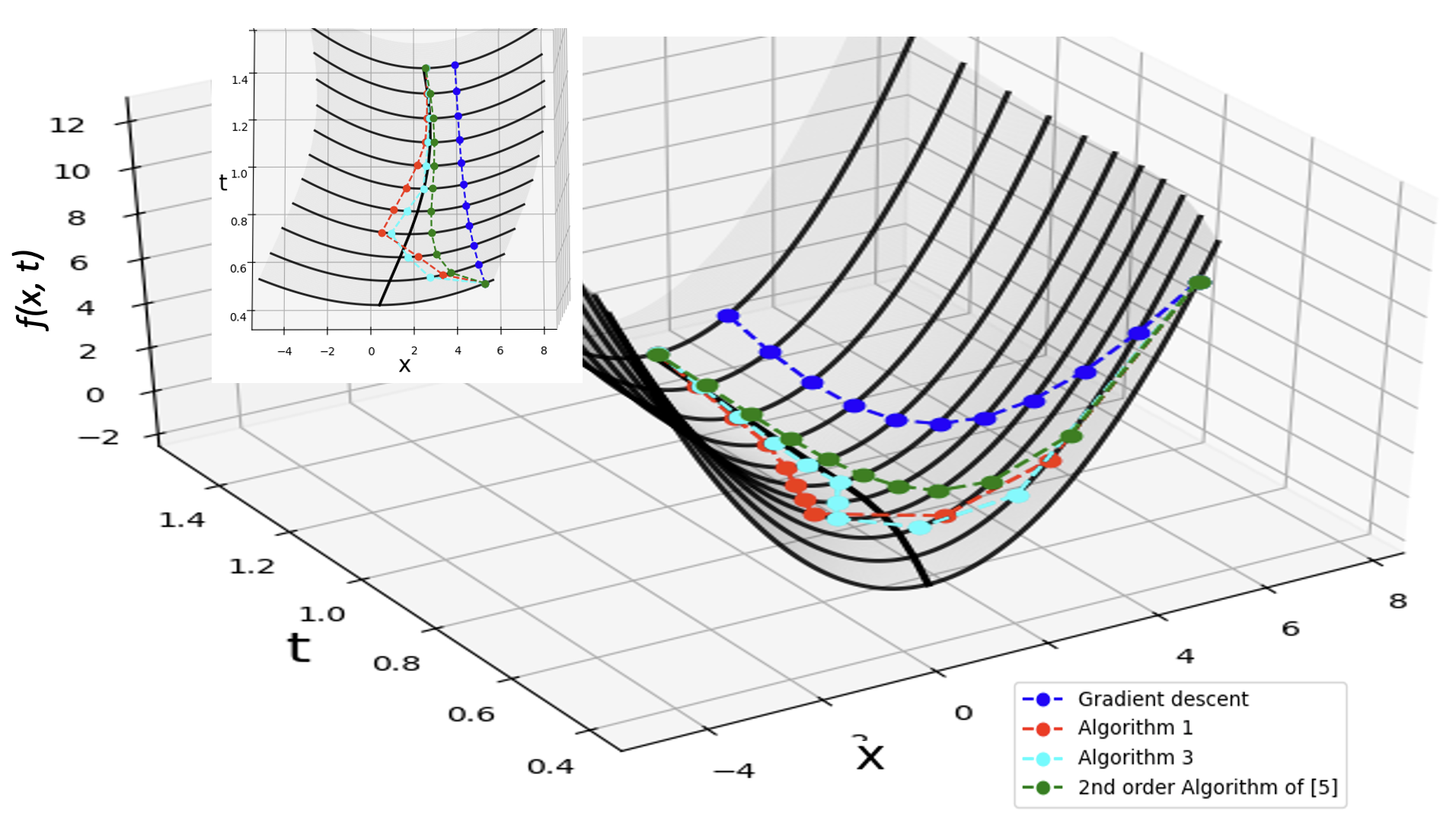}
      \caption{{\small The trajectories of the proposed Algorithms~\ref{Alg::1},~\ref{Alg::3} and the second-order Hessian-based algorithm of~\cite{AS-AM-AK-GL-AR:16} compared to the trajectory generated by the gradient descent algorithm shown in time interval $t\in[0.4, 1.4]$ for the cost function~\eqref{num1_cost}.}}
      \label{fig::ex1-fig1}
\end{figure}

  As depicted in Fig.~\ref{fig::ex1-fig1} and Fig.~\ref{fig::ex1-fig2}, the proposed Algorithms~\ref{Alg::1} and~\ref{Alg::3} show a better tracking behavior than the gradient descent algorithm, specially when $x_k$ is   far from the optimal value. is far from the optimal value. Notably, Algorithm~\ref{Alg::3}, as predicted by Lemma~\ref{lem::Alg3_vs_Alg1}, by incorporating more informative temporal knowledge, i.e., $\nabla_t f$ and $\nabla_{\vect{x}t} f$, achieves a lower tracking error and superior initial performance compared to other algorithms. Figures~\ref{fig::ex1-fig1} and \ref{fig::ex1-fig2} also display the trajectory of the optimal decision variable from the second-order Hessian-based algorithm in \cite{AS-AM-AK-GL-AR:16}, which, as expected, outperforms our proposed algorithms. Moreover, a noteworthy observation depicted in Fig.~\ref{fig::ex1-fig2} is that our first-order algorithms, particularly Algorithm~\ref{Alg::3}, outperform the second-order algorithm by~\cite{AS-AM-AK-GL-AR:16} when the tracking error is significant. Conversely, the second-order algorithm excels when the tracking error is small. This highlights the advantages of our hybrid Algorithm~\ref{Alg::4}: it achieves better tracking and reduces computational cost initially, when the error is large, and then switches to a second-order algorithm for more precise tracking as the error decreases.

 \begin{figure}[t]
  \centering
    \includegraphics[scale=0.24]{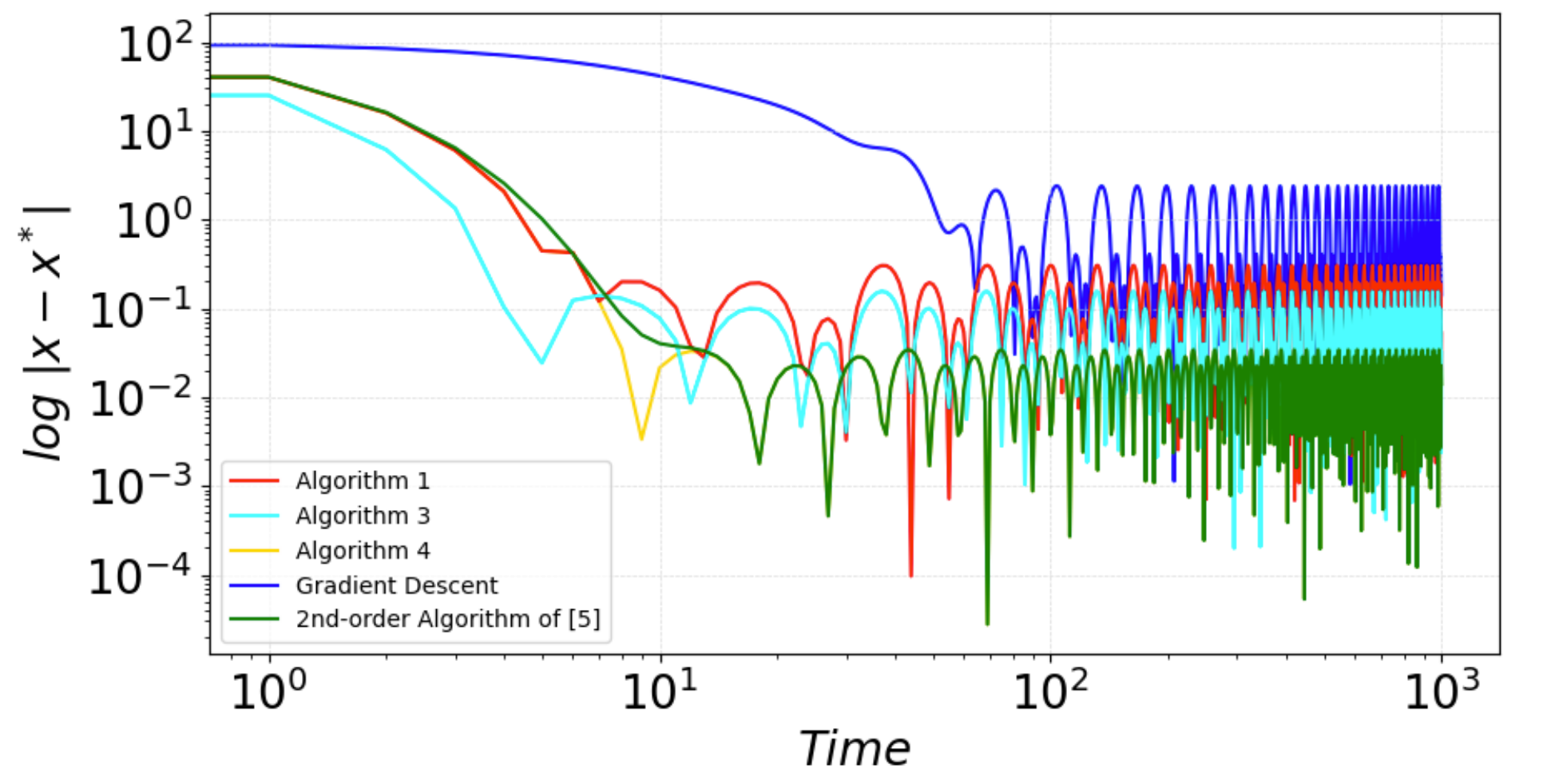}
      \caption{\small {Log error of the performance of Algorithm \ref{Alg::1}, 3 and \ref{Alg::4} versus gradient descent algorithm and the Hessian based  algorithm proposed in \cite{AS-AM-AK-GL-AR:16} with respect to the sampling time~$t_k$.}}
      \label{fig::ex1-fig2}
\end{figure}

\subsection{Model Predictive Control: a case of a time-varying cost whose temporal derivatives are not available explicitly} As a second demonstration, we consider an MPC problem cast as a time-varying optimization. MPC involves solving a sequence of optimization problems on some fixed moving horizon.  These iterative optimizations can be seen as instances of an optimization problem with time-varying costs. Let us demonstrate by designing an MPC controller for a unicycle robot shown in Fig.~\ref{fig:unicycle}. The control objective is to make point  $(x_h, y_h)$ as shown in Fig.~\ref{fig:unicycle}, which could, for example, denote the position of a camera installed on the robot, to trace or adhere to a desired trajectory $r(t) = [\begin{smallmatrix}
        r_x(t)& r_y(t)
\end{smallmatrix}]^\top.$ This desired trajectory is not predetermined; it is constructed or adapted dynamically at each time instance $t$ from perception module of the robot, for example, by observing the robot's path through the camera system within a forward-looking horizon. The dynamics governing the motion of point $(x_h,y_h)$ on the robot is given by (considering the geometry shown in Fig.~\ref{fig:unicycle}). Discretizing this dynamics we obtain $
    \left[\begin{smallmatrix}
        x_h(k+1) \\
        y_h(k+1)
    \end{smallmatrix}\right] = 
    \left[\begin{smallmatrix}
        x_h(k) + \delta u_1\\
        y_h(k) + \delta u_2
    \end{smallmatrix}\right].$ An MPC-based tracking controller can be obtained from solving the following optimization problem at each time step $k$, executing only controller $u(k)$ and repeating the~process:
\begin{subequations}\label{eq::MPC}
\begin{align}
    \min_{u_1\in \real^{10}} J_1(u_1(k)) &=  \sum\nolimits_{i=0}^{H_p-1}(r_x(k+i)-x_h(k+i))^2 \nonumber\\ 
    &+ \frac{1}{\lambda}\sum\nolimits_{i=0}^{H_u-1}(u_1(k+i))^2,\\
    \min_{u_2\in \real^{10}} J_2(u_2(k)) &=  \sum\nolimits_{i=0}^{H_p-1}(r_y(k+i)-y_h(k+i))^2 \nonumber\\
    &+ \frac{1}{\lambda}\sum\nolimits_{i=0}^{H_u-1}(u_2(k+i))^2.
\end{align}
\end{subequations}
  \begin{figure}[t]
  \centering
    \includegraphics[scale=0.22]{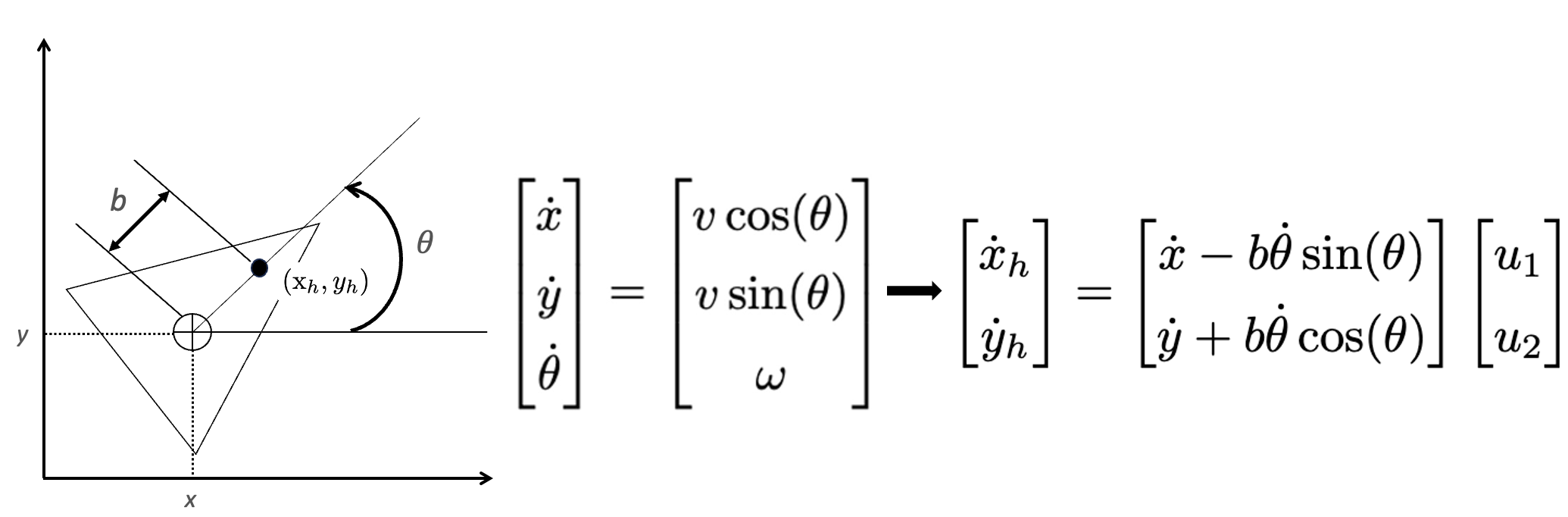}
      \caption{\small A unicycle robot and its equations of motion. }
      \label{fig:unicycle}
\end{figure} 
Here, $\lambda$ is the weight factor which is assumed to be constant on the prediction horizon. For our numerical example, we use the following values $\lambda = 0.1$, and same prediction horizon and the control horizon $H_p = H_u = 10$. Additionally, we initialize the robot from $x_h(0)=y_h(0)=0$ and set $u_1(i) = u_2(i) = 10, ~ \forall i \in \{0, 1,...,9\}$ when $k=0$. In addition, the path is generated by the function $y = \sin(\pi x), ~\forall x \in [-1,1]$.

A primary challenge with MPC control is the computational cost of repeatedly solving optimization problems like~\eqref{eq::MPC}. This often involves iterative methods, such as gradient descent, which converge asymptotically at each time step. Viewing the optimization problem \eqref{eq::MPC} as a time-varying cost at each time $t_k$, we solve this MPC problem using our proposed algorithms with parameters set to $\alpha = 0.5$, $\eps = 0.03$ and $\delta = 0.1$. Due to the lack of explicit knowledge of the cost's time derivative, we cannot use Algorithms~\ref{Alg::1} and \ref{Alg::4}. Thus, we use Algorithm~\ref{Alg::2} and a similar approximate version of Algorithm~\ref{Alg::4}.

The results in Fig.\ref{ex2} demonstrate that Algorithm~\ref{Alg::2} and the approximate version of Algorithm~\ref{Alg::4} achieve better tracking error than the gradient descent algorithm. This means they reach a specific error threshold with fewer steps, reducing computational cost. For instance, both these algorithms reach the 
$\epsilon$-error level at time $11.5$, while the gradient descent algorithm takes until time $33$, requiring $215$ more iterations. In Fig.\ref{ex2}, upon reaching the $\epsilon$ error threshold, Algorithm\ref{Alg::4} transitions to line 5 for the prediction step. However, due to the approximation for gradient values (i.e.,$\nabla_{\vect{x}t}J$ and $\nabla_{t}J$), it cannot completely eliminate the steady-state tracking error but achieves a lower error compared to Algorithm~\ref{Alg::2}. Note that in this problem, 
$n=10$, so a second-order algorithm must compute and invert a $10 \times 10$ Hessian matrix, which can be computationally costly.

\begin{figure}[t]
  \centering
    \includegraphics[scale=0.35]{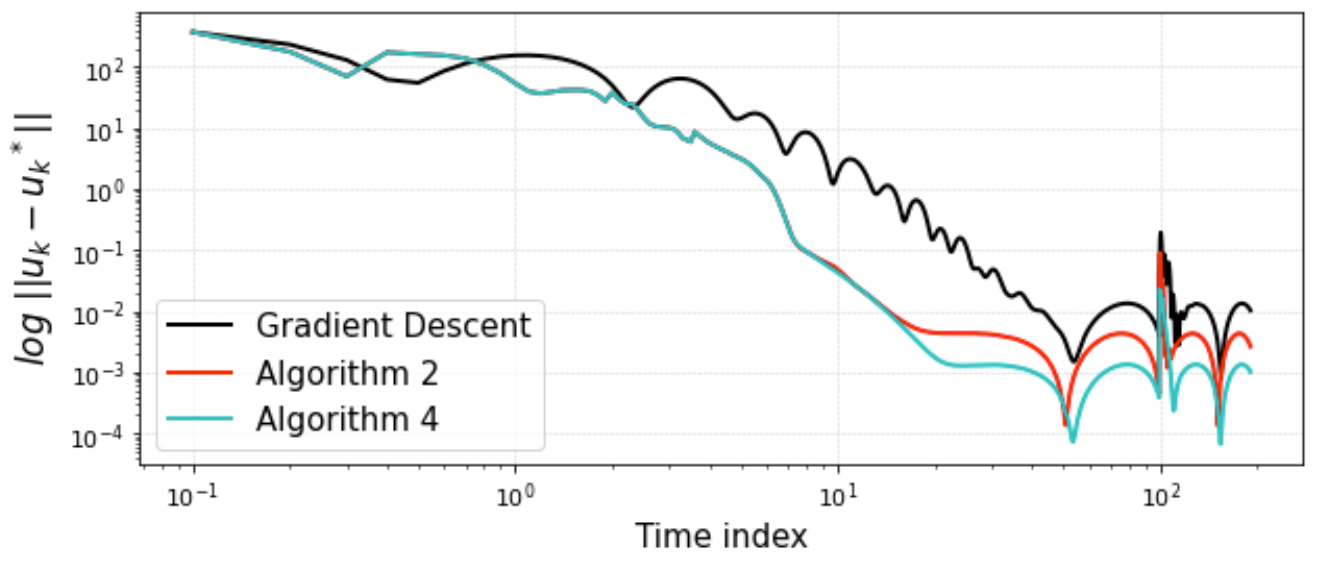}
      \caption{{\small Log error of the controller $u_1$ and $u_2$ with respect to time under executing Algorithm \ref{Alg::2}, approximated version of Algorithm \ref{Alg::4} and gradient descent algorithm. Here $u_k^{\star}$ is the optimal control value which given that the costs are quadratic can be obtained analytically by assuming enough computational power.}}
      \label{ex2}
\end{figure}

\section{Conclusion}\label{sec::conclu}
This paper introduced a set of first-order algorithms to solve a class of convex optimization problems with time-varying cost. We showed that these algorithms were able to follow the optimal minimizer trajectory with a bounded steady-state error. The primary motivation behind the development of these algorithms was to reduce the computational cost per time step from $O(n^3)$, as seen in existing Hessian-based algorithms, to $O(n)$, where $n$ represents the dimension of the decision variable. The algorithms proposed in this paper were inspired by the gradient descent algorithm but incorporate a prediction step, which allowed us to take into account the temporal variation of the cost to make sure that in transitioning from one time step to the other, the expected function reduction property is preserved across the time as well. Our proposed algorithms, which we referred to as $O(n)$-algorithms, were designed to solve the problem for different scenarios based on available temporal information about the cost. We also proposed a hybrid algorithm that switched to a second-order gradient tracking prediction step when the gradient got too close to zero and the $O(n)$ algorithms faced numerical problems due to the use of the inverse of the size of the gradient vector in their prediction step. Our numerical examples confirmed the improved tracking performance that was expected from this hybrid algorithm. Future work will focus on exploring the distributed implementation of the proposed algorithms for large-scale optimization problems in networked systems.

\bibliographystyle{ieeetr}%
\bibliography{bib/alias,bib/Reference}

\section*{Appendix}
In this appendix, we present the proof of our formal results. In these proofs invoke the properties of the cost that are explained in Section~\eqref{sec::prelim} and the implications of those assumptions, which are listed next. When $f$ is twice continuously differentiable, Assumption~\ref{asm:str_convexity} results in  $ m\, \vect{I}_n \preceq\nabla_{\vect{xx}}f(\vect{x},t),~\forall
~\vect{x}\in \real^n$~\cite{YN:18}, and Assumption~\ref{asm:M_Lip} results in $\nabla_{\vect{xx}} f(\vect{x},t)\preceq M\, \vect{I}_n,~\forall
~\vect{x}\in \real^n$~\cite{YN:18}. Under Assumptions~\ref{asm:str_convexity} and \ref{asm:M_Lip}, we also have~\cite{YN:18}
  \begin{subequations}
  \begin{align}
   &   \nabla_{\vect{x}} f(\vect{x},t)\!^\top\!(\vect{y}-\vect{x})+\frac{m}{2}\|\vect{y}-\vect{x}\|^2\leq f(\vect{y},t)-f(\vect{x},t)\nonumber\\
      &~~\qquad\quad\quad \leq \nabla_{\vect{x}} f(\vect{x},t)^\top(\vect{y}-\vect{x})+\frac{M}{2}\|\vect{y}-\vect{x}\|^2,\label{eq::convex_lipsc}\\
 &     \frac{1}{2M}\|\nabla f(\vect{x},t)\|^2\leq f(\vect{x},t)-f^\star(t) \leq \frac{1}{2m}\|\nabla f(\vect{x},t)\|^2,\label{eq::dist_optim_cost}
  \end{align}
  \end{subequations}
at each $t\in\real_{\geq0}$ and  for all $\vect{x},\vect{y}\in\real^n$.

\begin{proof} [Proof of Lemma~\ref{bound_traj_opt}] Under Assumptions~\ref{asm:str_convexity} and~\ref{asm:bound_dfstar}, ~\cite[Theorem 2F.10]{ALD-RTR:09} states that $t\mapsto\vectsf{x}^\star(t)$ satisfy 
\begin{align}\label{eq_bound_opt_traj}
\!\!\!    \|\vectsf{x}^\star_{k+1}\!-\vectsf{x}^\star_{k}\| \!\leq\! \frac{1}{m} \|\nabla_{\vect{x}t}f(\vect{x}_k,t_k)\| (t_{k+1}\!-t_k)\!\leq \!\frac{K_2\delta}{m}.~
\end{align}
Using Taylor series expansion~\eqref{eq::exact_taylor} with $(\vect{\chi}_{k+1},\vect{\chi}_k)=(\vectsf{x}^\star_{k+1},\vectsf{x}^\star_{k})$ and $(\tau_{k+1},\tau_k)=(t_{k+1},t_k)$, the proof follows from moving $f(\vectsf{x}^\star_k,t_k)$ to the left side of the expansion and upper bounding the right-hand side of the new equality using the bounds in Assumptions~\ref{asm:str_convexity} and \ref{asm:bound_dfstar} and the bounds in~\eqref{eq_bound_opt_traj}.
\end{proof}

\begin{proof}[Proof of Theorem~\ref{thm::main1}]
If $\|\nabla_{\vect{x}}f(\vect{x}_k,t_k)\|< \eps$, we have $\vect{x}_{k+1}^-=\vect{x}_k$. Then, given Assumption~\ref{asm:bound_dfstar} it follows from~\eqref{eq::exact_taylor} that 
\begin{align}\label{eq::grad_less_eps}
|f(\vect{x}^-_{k+1},t_{k+1})- f(\vect{x}_{k},t_{k})|\leq \mu\,\delta,\end{align}
where $\mu$ is given in the statement. If $\|\nabla_{\vect{x}}f(\vect{x}_k,t_k)\|\geq \eps$, substituting for $\vect{x}_{k+1}^-$ from Algorithm~\ref{Alg::1} (line $3$) in~\eqref{eq::exact_taylor} gives~\eqref{eq:f_pred}.
Under Assumption~\ref{asm:M_Lip}, Assumption~\ref{asm:bound_dfstar} and applying Cauchy–Schwarz inequality,~\eqref{eq:f_pred} leads to
\begin{align*}
    |f(\vect{x}^-_{k+1},t_{k+1})- &f(\vect{x}_{k},t_{k})| \leq \frac{\delta^2}{2}K_3+ 2\delta K_1 \\
&+\frac{\delta^2 K_1^2 M }{2\|\nabla_{\vect{x}} f(\vect{x}_{k},t_k)\|^2}
+\frac{\delta^2 K_1 K_2}{\| \nabla_{\vect{x}} f(\vect{x}_{k},t_k)\|},
\end{align*}

which along with $1/\|\nabla_{\vect{x}}f(\vect{x}_k,t_k)\|\leq 1/\eps$ results in $|f(\vect{x}^-_{k+1},t_{k+1})- f(\vect{x}_{k},t_{k})|\leq \gamma\delta^2$, where $\gamma$ is as given in the statement. 
Then, together with~\eqref{eq::grad_less_eps}, we can conclude that the prediction steps $2$ to $5$ of Algorithm~\ref{Alg::1} lead to
\begin{align} \label{eq::bound_alg1}
|f(\vect{x}^-_{k+1},t_{k+1})- f(\vect{x}_{k},t_{k})|\leq \max(\gamma\delta^2,\mu\delta).
\end{align}
For the update step, invoking~\eqref{eq::convex_lipsc}, we have 
\begin{align*}
&\nabla_{\vect{x}}  f(\vect{x}^-_{k+1},t_{k+1})^\top(\vect{x}_{k+1}-\vect{x}^-_{k+1})+\frac{m}{2}\|\vect{x}_{k+1}-\vect{x}^-_{k+1}\|^2\\&\leq f(\vect{x}_{k+1},t_{k+1})-f(\vect{x}^-_{k+1},t_{k+1})\leq\\&\nabla_{\vect{x}}  f(\vect{x}^-_{k+1},t_{k+1})^\top(\vect{x}_{k+1}-\vect{x}^-_{k+1})+\frac{M}{2}\|\vect{x}_{k+1}-\vect{x}^-_{k+1}\|^2,
\end{align*}
 which along with update line $7$ of the Algorithm~\ref{Alg::1}~gives
\begin{align*}
  &-\alpha (1-\frac{m}{2}\alpha)\|\nabla_{\vect{x}}  f(\vect{x}^-_{k+1},t_{k+1})\|^2\leq f(\vect{x}_{k+1},t_{k+1})-\\
&\quad f(\vect{x}^-_{k+1},t_{k+1})\leq -\alpha (1-\frac{M}{2}\alpha)\|\nabla_{\vect{x}}   f(\vect{x}^-_{k+1},t_{k+1})\|^2.
\end{align*}
Next, note that given Assumption~\ref{asm:str_convexity} and in light of~\eqref{eq::dist_optim_cost}, we have 
 \begin{align*}
  \!\! -2M (&f(\vect{x}^-_{k+1},t_{k+1})-f^\star(t_{k+1}))\!\leq\!\!-\|\nabla_{\vect{x}}  f(\vect{x}^-_{k+1},t_{k+1})\|^2\\&  \!\leq -2m (f(\vect{x}^-_{k+1},t_{k+1})-f^\star(t_{k+1})).
\end{align*}
Let us write $ f(\vect{x}_{k+1},t_{k+1})-f^\star(t_{k+1})
   \!=\!f(\vect{x}_{k+1},t_{k+1})-f(\vect{x}^-_{k+1},t_{k+1})+(f(\vect{x}^-_{k+1},t_{k+1})-f^\star(t_{k+1})),$
which together with the bounds we already derived for the update step leads to
\begin{align}\label{eq::bound_int}
   & f(\vect{x}_{k+1},t_{k+1})\!-\!f^\star(t_{k+1})\leq\nonumber\\
    &\quad \quad\quad  (1-2\kappa  \alpha m)\,(f(\vect{x}^-_{k+1},t_{k+1})-f^\star(t_{k+1})),
\end{align}
where $\kappa=(1-\alpha M/2)$. Here, given that $0<m\leq M$, we used the fact that for $0<\alpha\leq \frac{1}{2M}$ we have $0<2\alpha m(1-\alpha M/2)<2\alpha M(1-\alpha m/2)\leq 1$.
 On the other hand, by taking into account~\eqref{eq::bound_alg1} and \eqref{eq::bound_optimal_trajectory} we can write
\begin{align*}
&f(\vect{x}^-_{k+1},t_{k+1})-f^{\star}(t_{k+1})=f(\vect{x}^-_{k+1},t_{k+1})-f(\vect{x}_{k},t_{k})\\&+f(\vect{x}_{k},t_{k})-f^{\star}(t_{k})+f^{\star}(t_{k})-f^{\star}(t_{k+1})\leq \\
&\max(\gamma\delta^2,\mu\delta)+f(\vect{x}_{k},t_{k})-f^{\star}(t_{k})+ \psi.
\end{align*}
Then, from~\eqref{eq::bound_int} we can obtain  
 \begin{align*}
&f(\vect{x}_{k+1},t_{k+1})-f^{\star}(t_{k+1}) \leq(1-2\kappa \alpha m )\max(\gamma\delta^2,\mu\delta)
\\&+
(1-2\kappa \alpha m )(f(\vect{x}_{k},t_{k})-f^{\star}(t_{k}))+(1-2\kappa \alpha m )\psi.
\end{align*}
Subsequently, since $0<1-2\kappa \alpha m\leq 1$ we can obtain~\eqref{eq::bound_on_solu_1}, completing the proof.
\end{proof}

\begin{proof}[Proof of Theorem~\ref{thm::main2}]
First note that if $\|\nabla_{\vect{x}}f(\vect{x}_k,t_k)\|< \eps$ Algorithm~\ref{Alg::2} also results in~\eqref{eq::grad_less_eps}. If $\|\nabla_{\vect{x}}f(\vect{x}_k,t_k)\|\geq  \eps$, we proceed as follows. Note that
\begin{align}\label{eq::taylor_proof}
    &f(\vect{x}_k,t_{k-1})=f(\vect{x}_k,t_{k})-\nabla_t f(\vect{x}_k,t_k)(t_{k}-t_{k-1})\nonumber\\&+\frac{1}{2}\nabla_{tt} f(\vect{x}_k,\nu)(t_{k}-t_{k-1})^2,~\nu\in[t_{k-1},t_k).
\end{align}
 Considering~\eqref{eq::taylor_proof}, if $\|\nabla_{\vect{x}}f(\vect{x}_k,t_k)\|\geq  \eps$, the prediction step of Algorithm~\ref{Alg::2} (line 3) results in
\begin{align*}
 &f(\vect{x}^-_{k+1},t_{k+1})- f(\vect{x}_{k},t_{k})=\frac{\delta^2}{2}\nabla_{tt} f(\vect{\vect{\zeta}},\theta) + \delta \nabla_{t}f(\vect{x}_{k},t_k)\\
&- \delta \big|\nabla_{t}f(\vect{x}_{k},t_k) - \frac{\delta}{2}\nabla_{tt} f(\vect{x}_{k},\nu))\big|\\
& +\frac{\delta^2 \big|\nabla_{t}f(\vect{x}_{k},t_k) - \frac{\delta}{2}\nabla_{tt} f(\vect{x}_{k},\nu))\big|^2}{2\|\nabla_{\vect{x}} f(\vect{x}_{k},t_k)\|^4}\\
&\nabla_{\vect{x}} f(\vect{x}_{k},t_k)^\top \nabla_{\vect{xx}}f(\vect{\vect{\zeta}},\theta)\nabla_{\vect{x}} f(\vect{x}_{k},t_k)\\
&-\frac{\delta^2 |\nabla_t f(\vect{x}_{k},t_k)-\frac{\delta}{2}\nabla_{tt} f(\vect{x}_{k},\nu)|}{\| \nabla_{\vect{x}}f(\vect{x}_{k},t_k)\|^2}  \nabla_{\vect{x}t} f (\vect{\vect{\zeta}},\theta)^\top \nabla_{\vect{x}} f(\vect{x}_{k},t_k),
\end{align*}
which given Assumption \ref{asm:M_Lip} and ~\ref{asm:bound_dfstar} and applying Cauchy–Schwarz inequality and using the fact that $|a+b|^2 \leq 2|a|^2 + 2|b|^2$, $\forall$ $a,b \in \real$ we have 
\begin{align*}
&|f(\vect{x}^-_{k+1},t_{k+1})- f(\vect{x}_{k},t_{k})| \leq \delta^2K_3 + 2\delta K_1\\
& +\frac{\delta^2 K_1^2 M}{\|\nabla_{\vect{x}} f(\vect{x}_{k},t_k)\|^2} +\frac{\delta^2 K_2 (K_1 + \frac{\delta}{2}K_3)}{\| \nabla_{\vect{x}}f(\vect{x}_{k},t_k)\|} \\
&+\frac{\delta^4 K_3^2 M}{4\|\nabla_{\vect{x}} f(\vect{x}_{k},t_k)\|^2},
\end{align*}
which together with $1/\|\nabla_{\vect{x}}f(\vect{x}_k,t_k)\|\leq 1/\eps$ results in $|f(\vect{x}^-_{k+1},t_{k+1})- f(\vect{x}_{k},t_{k})|\leq \gamma'\delta^2$.
Thus, given 
Assumption~\ref{asm:bound_dfstar} and considering~\eqref{eq::grad_less_eps}, we can conclude that the prediction step (lines $2$ to $5$) of Algorithm~\ref{Alg::2} lead~to
\begin{align}\label{eq::bound_alg11}
|f(\vect{x}^-_{k+1},t_{k+1})- f(\vect{x}_{k},t_{k})|\leq \max(\gamma'\delta^2,\mu\delta),
\end{align}
where $\gamma'$ is given in the statement. For the update step, we follow the similar approach as the proof Algorithm~\ref{Alg::1} to arrive as same inequality relation~\eqref{eq::bound_int}.
On the other hand, we have $f(\vect{x}^-_{k+1},t_{k+1})-f^{\star}(t_{k+1})=f(\vect{x}^-_{k+1},t_{k+1})-f(\vect{x}_{k},t_{k})+f(\vect{x}_{k},t_{k})-f^{\star}(t_{k})+f^{\star}(t_{k})-f^{\star}(t_{k+1})$, which along with 
invoking~\eqref{eq::bound_alg11}, \eqref{eq::bound_optimal_trajectory} and~\eqref{eq::bound_int} results in
\begin{align*}
&f_2(t_{k+1})-f^{\star}(t_{k+1})\leq(1-2\kappa\alpha m)\max(\gamma'\delta^2,\mu\delta)
\\&+
(1-2\kappa\alpha m)(f(\vect{x}_{k},t_{k})-f^{\star}(t_{k}))+(1-2\kappa\alpha m)\psi,
\end{align*}
and consequently~\eqref{eq::bound_on_solu_2},
which concludes our proof.
\end{proof}

\begin{proof}[Proof of Theorem~\ref{thm::main1::alg3}]
If $\|\nabla_{\vect{x}}f(\vect{x}_k,t_k) + \delta\nabla_{\vect{x}t}  f(\vect{x}_k,t_k)\|\geq \eps$ and 
$\nabla_{\vect{x}t}  f(\vect{x}_k,t_k)^\top \nabla_{\vect{x}}  f(\vect{x}_k,t_k) \leq 0$, then, by adding and subtracting $\nabla_{\vect{x}t} f(\vect{x}_k,t_k) (\vect{x}^-_{k+1}-\vect{x}_{k}) \delta$ and substituting for $\vect{x}_{k+1}^-$ from Algorithm~\ref{Alg::3} (line $3$) in \eqref{eq::exact_taylor} gives \eqref{eq::pred_alg3}. Under Assumption~\ref{asm:M_Lip}, Assumption~\ref{asm:bound_dfstar} and applying Cauchy–Schwarz inequality, and along with $1/\|\nabla_{\vect{x}} f(\vect{x}_{k},t_k) + \delta\nabla_{\vect{x}t}  f(\vect{x}_k,t_k)\|\leq 1/\eps$, ~\eqref{eq::exact_taylor} leads to
\begin{align*}
&|f(\vect{x}^-_{k+1},t_{k+1})- f(\vect{x}_{k},t_{k})|\leq 2\delta K_1 + \frac{\delta^2 K_1^2 M}{2 \eps^2}\\ 
&+ \frac{\delta^2}{2}K_3 + \frac{2\delta^2 K_1 K_2}{\eps},
\end{align*}
which results in $|f(\vect{x}^-_{k+1},t_{k+1})- f(\vect{x}_{k},t_{k})|\leq \eta\delta^2$, where $\eta$ is given in the statement. 
Furthermore, same as the proof in Algorithm \ref{Alg::1}, the other cases of Algorithm \ref{Alg::3} will result in \begin{align}\label{eq::grad_less_eps:alg3}
|f(\vect{x}^-_{k+1},t_{k+1})- f(\vect{x}_{k},t_{k})|\leq \max(\gamma\delta^2,
\mu\delta),
\end{align}
where $\mu$ and $\gamma$ are given in the statement.

Therefore, taking into account~\eqref{eq::grad_less_eps:alg3}, we can conclude that the prediction steps $2$ to $5$ of Algorithm~\ref{Alg::3} lead to
\begin{align}\label{eq::bound_alg3}
|f(\vect{x}^-_{k+1},t_{k+1})- f(\vect{x}_{k},t_{k})|\leq \max(\gamma\delta^2,\mu\delta, \eta\delta^2).
\end{align}
For the update step, similar to the proof for Algorithm \ref{Alg::1}, we obtain  
 \begin{align*}
&f(\vect{x}_{k+1},t_{k+1})-f^{\star}(t_{k+1}) \leq(1-2\kappa \alpha m )\max(\gamma\delta^2,\mu\delta, \eta \delta^2)
\\&+
(1-2\kappa \alpha m )(f(\vect{x}_{k},t_{k})-f^{\star}(t_{k}))+(1-2\kappa \alpha m )\psi.
\end{align*}
Subsequently, since $0<1-2\kappa \alpha m\leq 1$ we can obtain~\eqref{eq::bound_on_solu_5}, completing the proof.
\end{proof}

\begin{proof}[Proof of Theorem~\ref{thm6}]
If $\|\nabla_{\vect{x}}f(\vect{x}_k,t_k)\|\geq \eps$, following the same steps in the proof of Theorem~\ref{thm::main1}, we can show that the implementation of Algorithm~\ref{Alg::4} results in~\eqref{eq::grad_less_eps}. If 
 $\|\nabla_{\vect{x}}f(\vect{x}_k,t_k)\|< \eps$, under Assumption ~\ref{asm:M_Lip}, Assumption~\ref{asm:bound_dfstar}, we obtain
\begin{align}\label{eq::grad_less_eps_6}
|f(\vect{x}^-_{k+1},t_{k+1})- f(\vect{x}_{k},t_{k})|\leq \mu\,\delta,\end{align}
 from using the Taylor series explanation~\eqref{eq::exact_taylor} for $f(\vect{x}^-_{k+1},t_{k+1})$ and substituting for $\vect{x}_{k+1}^-$ from Algorithm~\ref{Alg::4}'s line $5$; similar to the procedure used in the proof of Theorem~\ref{thm::main1}.
The remainder of the proof also follows exactly the same steps as in the proof of Theorem~\ref{thm::main1}, and is not repeated here for the sake of~brevity.
\end{proof}

\end{document}